\documentclass[11pt,twoside,english]{article}
\usepackage{amsmath}
\usepackage{amsfonts}
\usepackage{mathrsfs}
\usepackage{color}
\usepackage{ amsmath, amsfonts, amssymb, amsthm, amscd}
\usepackage[T1]{fontenc}
\usepackage[english]{babel}
\usepackage[noinfoline]{imsart}

\newtheorem{theorem}{Theorem}
\newtheorem{lemma}{Lemma}
\newtheorem{proposition}{Proposition}

\newtheorem{definition}{Definition}
\newtheorem{corollary}{Corollary}

\newcommand{\cH}{\ensuremath{\mathcal H}}

\newcommand{\cO}{\ensuremath{\mathcal O}}

\newcommand{\bbR}{{\ensuremath{\mathbb R}} }

\newcommand{\be}{\begin{equation}}
\newcommand{\ee}{\end{equation}}
\newcommand{\beq}{\begin{eqnarray}}
\newcommand{\eeq}{\end{eqnarray}}
\newcommand{\R}{\mathbb{R}}

\newcommand{\ced}{\end{proof}}

\newcommand{\p}{^{\prime}}

\begin{document}

\begin{frontmatter}
\title{The Obstacle Problem for  Quasilinear Stochastic Integral-Partial Differential Equations}
\date{}
\runtitle{}
\author{\fnms{Yuchao}
	\snm{DONG}\corref{}\ead[label=e1]{ycdong@fudan.edu.cn}}
%\thankstext{T1}{The first auther gratefully acknowledges finincial support from R\'egion Pays de la Loire throught the grant PANORisk }
\address{Fudan University and University of Angers
	\\\printead{e1}}
\author{\fnms{Xue}
	\snm{YANG}\corref{}\ead[label=e2]{xyang2013@tju.edu.cn}}
%\thankstext{T1}{The work of the second author is supported by National Natural Science Foundation of China (11401427) }
\address{Tianjin University
	\\\printead{e2}}
\author{\fnms{Jing}
	\snm{ZHANG}\corref{}\ead[label=e3]{zhang\_jing@fudan.edu.cn}}
%\thankstext{T2}{The work of the third author is supported by National Natural Science Foundation of China (11401108) and Shanghai Science and Technology Commission Grant (14PJ1401500).}
\address{Fudan University
	\\\printead{e3}}

\runauthor{Y. Dong, X. Yang and J. Zhang}

\begin{abstract}
We prove an existence and uniqueness result for the obstacle problem for quasilinear stochastic integral-partial differential equations. Our method is based on the probabilistic interpretation of the solution using backward doubly SDEs with jumps. 
\end{abstract}


\end{frontmatter}

\section{Introduction}

We consider the following stochastic partial differential equations with obstacles (OSPDEs for short)
\small
\begin{equation}\label{SPDE1} 
\left\{
\begin{split}
&du_t (x) + [ \frac{1}{2} \mathcal A u_t (x)  + f_t(x,u_t (x),\nabla u_t (x))]\, dt+  h_t(x,u_t(x),\nabla u_t(x))\cdot \overleftarrow{dB}_t  = 0, \, (t,x)\in [0,T]\times R^d;\\
 & u_t(x)\geq v_t(x),\quad  (t,x)\in [0,T]\times R^d;\\
 & u_T(x) = \Phi(x), \quad x\in R^d.
\end{split}\right.
\end{equation}
The operator $\mathcal A$, which is non-local,  is an symmetric infinitesimal generator of a Markov process with jumps.   $f$ and $h = \big(h_1, \cdots,h_{d^1}\big)$ are non-linear random functions. The differential term with $\overleftarrow{dB}_t$
refers to the backward stochastic integral with respect to a $d^1$-dimensional Brownian motion on a probability space
$\big(\Omega, \mathcal{F},\mathbb{P} \big)$, so that the doubly stochastic framework introduced by Pardoux and Peng \cite{PardouxPeng94} could be applied. Given an obstacle $v: \Omega\times[0,T]\times \cO \rightarrow \R$, we study the OSPDE \eqref{SPDE1},
i.e. we want to find a solution that satisfies "$u\geq v$" where the obstacle $v$ is regular in some sense.

Nualart and Pardoux \cite{NualartPardoux} have studied the obstacle problem for a nonlinear heat equation on the spatial interval
$[0,1]$ driven by a space-time white noise with the diffusion matrix $a=I$. 
Then  Donati-Martin and Pardox \cite{Donati-Pardoux} proved it for the general diffusion matrix.  
Various properties of  the solutions were studied later in \cite{DMZ}, \cite{XuZhang}, \cite{Zhang} etc.. 
Denis, Matoussi and Zhang (\cite{DMZ12})  study the OSPDE  with null Dirichlet condition on an open domain in $\mathbb{R}^d$. 
Their method is based on the techniques of parabolic potential theory developed by M. Pierre (\cite{Pierre}, \cite{PIERRE}). 
The solution is expressed as a pair $(u,\nu)$  where $u$ is a predictable
continuous process which takes values in a proper Sobolev space and
$\nu$ is a random regular measure satisfying  some minimal 
condition. The key point was to construct a  solution $u$ which admits a quasi-continuous version defined outside a polar set and the regular measures $\nu$ which in general are not absolutely continuous w.r.t. the Lebesgue measure. 
\vspace{2mm}

As backward stochastic differential equations (BSDEs for short) was rapidly developed, obstacle problem associated with a non-linear partial differential equation 
(PDE for short) with more general coefficients, and the properties of the solutions for OSPDEs were studied in the framework of BSDEs (\cite{BCEF}, \cite{Elk2}, etc.).
Matoussi and Stoica  \cite{MS10} have proved an existence and
uniqueness result for the obstacle problem of backward quasilinear
stochastic PDE on the whole space $\bbR^d$ and driven by a finite dimensional Brownian motion. The method is based on the probabilistic
interpretation of the solution by using the backward doubly stochastic differential equation (BDSDE in short).  The solution also is expressed as a pair $( u,\nu)$.
%where  $u$ is a predictable continuous process which takes values in a proper  Sobolev space 
%and $\nu$ is a random regular measure satisfying minimal Skohorod condition. 
It is essential to give the regular measure $\nu$  a probabilistic interpretation in term of the continuous increasing  process $K$ where $ (Y,Z,K)$ is the solution of a  reflected generalized BDSDE.
\vspace{2mm}

The stochastic partial differential equations in \eqref{SPDE1} without the obstacle was studied in \cite{Denis2}, which is the main motivation of this article.  It inspires us to generalize the obstacle problem in \cite{DMZ12}  and consider a more general  linear  integral-differential operator.  Reflected BSDE with jumps, which is a standard reflected BSDE driven by a Brownian motion and an independent Poisson point process, has been studied by Hamad\`ene and Ouknine in \cite{HO}.  After that, parabolic integro-differential partial equations with two obstacles are solved by Harraj, Ouknine and Turpin in \cite{HOT}. But in the framework of BSDEs, they concern on viscosity solutions for obstacle problems.
\vspace{2mm}

Our aim is to study the OSPDE \eqref{SPDE1} with a non-negative, self-adjoint operator associated with a symmetric L\'evy process by using the probabilistic method in \cite{MS10} and prove the existence and uniqueness of the weak solution. The quasi-continuity of $u$ makes it possible for us to find the regular measure $\nu$ satisfying $\nu(u>v)=0$.  In this paper,  the difficulty mainly lies to the estimate on the jump part of the l\'evy process  during the approximation so that the uniform convergence of the penalization sequence on the trajectories can be obtained.  But in the present work, the model does not contain the term of divergence as in \cite{MS10},  since the   L\'evy process is considered here,  and we will leave this problem into the future work.
\vspace{2mm}

The remainder of this paper is organized as follows: in the second section, we set the function spaces,  probability space and  introduce the notion of regular measures associated to parabolic potentials. The quasi-continuity of the solution for SPDE without obstacle is also proved in this section. The third section is devoted to proving the existence and uniqueness result.

\section{Preliminaries}
\label{spaces}
Let $L^2(\mathbb{{R}}^d) $ be the set of square integrable functions with respect to Lebesgue measure on $\mathbb{R}^d$. It is a Hilbert space equipped with the usual  scalar product and norm as follows, 
$$(u,v) =\int_{\mathbb{R}^d}u(x)v(x)dx,\quad \|u\| _2=\left(\int_{\mathbb{R}^d}u^2(x) dx\right)^{\frac 12}. $$ 
We also use the notation
\[ (u,v)=\int_{\mathbb{R}^d} u(x)v(x)\, dx,\]
where $u$, $v$ are measurable functions defined in $\mathbb{R }^d$ and $uv \in L^1 (\mathbb{R}^d )$. We will consider an evolution problem over a fixed time interval $[0,T]$ and define the norm for a function in $L^2( [0,T] \times \mathbb{{R}}^d)$ as 
$$\left\| u\right\| _{2,2}=\left(\int_0^T  \int_{\mathbb{R}^d} |u (t,x)|^2 dx dt \right)^{\frac 12}. $$

Another Hilbert space that we use is the first order Sobolev space $H^1(\mathbb{R }^d)$. Its natural scalar product and norm are
$$\left( u,v\right) _{H^1\left( {\mathbb{R}^d}\right) }=\left( u,v\right) +\left( \nabla u, \nabla v \right),\quad \left\| u\right\| _{H^1\left(
{\mathbb{R}^d}\right) }=\left( \left\| u\right\| _2^2+\left\| \nabla u\right\| _2^2\right) ^{\frac 12},$$
where $\nabla u (t,x) = (\partial_1 u (t,x), \cdots, \partial_d u (t,x))$ denotes the gradient.

Of special interest is the subspace $\widetilde{F}  \subset L^2( [0,T]; H^1( {\mathbb{R}^d})) $ consisting of all functions $u(t,x)$ such that 
$t \mapsto u_t = u(t, \cdot)$ is continuous in $L^2 (\mathbb{R}^d)$. The natural norm on $\widetilde{F}$ is
$$\left\| u\right\| _{T}= \sup_{ 0 \leq t \leq T}\left\| u_t\right\|_2 + \left(\int_0^T  \|\nabla u_t \|_2 dt \right)^{\frac 12}.$$

The space of test functions is denoted by
 $ \mathcal{D}_T  = \mathcal{C}^{\infty} \big([0,T]\big) \otimes \mathcal{C}_c^{\infty} \big(\mathbb{R}^d\big)$, where
$\mathcal{C}^{\infty} \big([0,T]\big)$ denotes the space of real functions which can be extended as infinite differentiable functions in the neighborhood of $[0,T]$ 
and $ \mathcal{C}_c^{\infty}\big(\mathbb{R}^d\big)$ is the space of infinite differentiable functions with compact support in $\mathbb{R}^d$.
%The Lebesgue measure in $\mathbb{R}^d$ will be sometimes denoted by $m$.

%%%%%%%%%%%%%%%%%%%%%%%%%%%%%%%%%%%%%%%%%%%%%%%%%%%%%%%%%%%%%%%%%%%%%%%%%%%%
\subsection {The corresponding Markov process}

We consider a Dirichlet form $(\mathcal E, H^1(\mathbb{ R}^d))$ on $L^2(\mathbb{ R}^d)$ as follows: 
\[\mathcal E(u,v)=\mathcal {E}^1(u,v)+\mathcal {E}^2(u,v)\]
with
\[\mathcal {E}^1(u,v)=\frac{1}{2}(\nabla u,\nabla v)\quad \mbox{and}\quad\mathcal {E}^2(u,v)=\frac{1}{2}\int_{\mathbb{ R}^d\times\mathbb{ R}^d\setminus \Gamma}\frac{(u(x)-u(y))(v(x)-v(y))}{|x-y|^{d+\alpha}}dxdy,\]
where $\Gamma$ is the diagonal $\Gamma:= \{(x,x)|x\in \mathbb{R}^d\}$ and  $\alpha \in (0,2)$.

From classical probability theory, we know that such a Dirichlet form is related to a Hunt process, whose infinitesimal generator $\mathcal A$ is 
\[\mathcal A u(x)=\frac{1}{2}\Delta u(x) +\int_{\mathbb{ R}^d}\big(u(x+y)-u(x)\big )\frac{dy}{|y|^{d+\alpha}}.\]
Let $\Omega^1 = \mathcal{C }\left([0, \infty ); \mathbb{R}^d \right)$ be the space of continuous trajectories. The canonical process $(W_t)_{t \geq 0}$ is defined by $ W_t (\omega^1) = \omega^1 (t)$, for any $ \omega^1 \in \Omega^1$, $t \geq 0$ and  the shift operator, $ \theta_t  \, : \,  \Omega^1 \longrightarrow  \Omega^1$, is defined by $ \theta_t (\omega^1) (s) = \omega^1 (t+s)$, for any $s,\,t\geq 0$. The canonical filtration $ \mathcal{F}_t^1 = \sigma \left( W_s; s \leq t \right)$ is completed by the standard procedure with respect to the probability measures produced by the transition function
$$Q_t (x, dy) = q_t (x-y) dy, \quad t >0, \quad x \in \mathbb{R}^d,$$
where $ q_t (x) = \left(2\pi t\right)^{- \frac{d}{2}} \exp \left( - |x|^2/2t \right)$ is the Gaussian density. Thus, we get a continuous Hunt process $\left(\Omega^1, W_t, \theta^1_t, \mathcal{F}^1, \mathcal{F}^1_t, \mathbb{P}_1^x \right)$. 
%We shall also use the backward filtration of the future events $ \mathcal{F}'_t = \sigma \left(W_s; \; \,  s \geq t \right)$ for $t\geq 0$. 
$\mathbb{P}_1^0$ is the Wiener measure, which is supported by the set $ \Omega^1_0 = \{ \omega^1 \in \Omega^1, \; \,  w^1 (0) =0 \}$. We set $ \Pi_0 (\omega^1) (t) = \omega^1 (t) - \omega^1 (0),\,  t \geq 0$, which defines a map $ \Pi_0 \, : \, \Omega^1 \rightarrow \Omega^1_0$. Then  $\Pi = (W_0, \Pi_0 ) \, : \,  \Omega^1 \rightarrow \mathbb{R}^d \times \Omega^1_0$ is a bijection. For each probability measure on $\mathbb{R}^d$, the probability $\mathbb{P}^{\mu}$ of the Brownian motion started with the initial distribution $\mu$ is given by $$ \mathbb{P}_1^{\mu} = \Pi^{-1} \left(\mu \otimes \mathbb{P}_1^0 \right).$$
In particular, for the Lebesgue measure in $\mathbb{R}^d$, which we denote by $ m = dx$, we have
$$ \mathbb{P}_1^{m} = \Pi^{-1} \left(dx\otimes \mathbb{P}_1^0 \right).$$
%%These relations are saying that $W_0$ is independent of $ \Pi_0$. It is known that each component $(W^i_t)_{t \geq 0}$ of the Brownian motion, $ i =1,\cdots,d$, is a martingale under any of the measures $\mathbb{P}_1^{\mu}$. The next lemma shows that
%%$ \left(W_{t-r}^i, \mathcal{F}'_{t-r}\right)$, $r \in (0, t]$ is  a backward local martingale under $\mathbb{P}_1^{m}$.
%%\begin{lemma}(Lemma 1 in \cite{MS10}) Let $ 0< s < t$. If $A \in \sigma (W_t)$ is such that $ \mathbb{E}_1^m \left[|W_t|; A \right] < \infty$, then one has   $ \mathbb{E}^m \left[|W_s|; A \right] < \infty$. Moreover, for each $ B \in \mathcal{F}'_{t}$, and $ i = 1, \cdots, d$, one has $$
%%  \mathbb{E}_1^m \left[W^i_s; A \cap B  \right] = \mathbb{E}_1^m \left[W^i_t; A \cap B \right].$$
%%\end{lemma}
 
Denote $\Omega^2 :=D([0,T];\mathbb{ R}^d)$ as the Skorohod space. The canonical processs $V_t$ and the shift operator $\theta^2_t$ can be defined similarly to $W_t$ and $\theta^1_t$ given above. Hence, we get a Hunt process $(\Omega^2,V_t,\theta^2_t, \mathcal {F}^2,\mathcal {F}_t^2,\mathbb{ P}_2^x)$ related to the Dirichlet form $\mathcal {E}^2$. 

We consider the sample space $\Omega'=\Omega^1 \times \Omega^2$ and the process $(X_t)_{t\geq0}$ defined by $X_t(\omega^1,\omega^2)=W_t(\omega^1)+V_t(\omega^2)$ for $t\geq0$. The shift operator $\Theta_t:\Omega' \longleftarrow \Omega'$ is defined by $\Theta_t(w^1,w^2)(s)=(w^1(t+s),w^2(t+s))$, for any $s,\,t\geq 0$. The $\sigma$-field $\mathcal F$ and filtration $\mathcal {F}_t$ are given by $\mathcal F:=\mathcal {F}^1 \times \mathcal {F}^2$ and $\mathcal {F}_t=\mathcal {F}^1_t \times \mathcal {F}^2_t$. The family of probability measures $\{\mathbb{ P}^x\}_x$ is defined by $\mathbb{ P}^x:=\mathbb{ P}_1^x \times \mathbb{ P}_2^0$. Therefore, we see that $(\Omega',X_t,\Theta_t,\mathcal F,\mathcal {F}_t,\mathbb{P}^x)$ is a homogeneous Markov process related to the Dirichlet form $(\mathcal E,H^1(\mathbb{ R}^d))$. For the process $X$, we have the following decomposition: 
\[X_t=W_t+\int_0^t\int_{|z|\ge1}zN(dz,dt)+\int_0^t\int_{|z|<1}z\tilde N(dz,dt),\]
where $N(dz,dt)$ is the jumping measure of $X$ with $v(dz)dt:=\frac{1}{|z|^{d+\alpha}}dzdt$ as its predictable compensator and $\tilde N(dz,dt):=N(dz,dt)-v(dz)dt$ the associated compensated measure. 

Denote by $P_t$ the corresponding semigroup which is strongly continuous on $L^2(\mathbb{ R}^d)$. It is easy to verify that the transition function is absolutely continuous with respect to the Lebseque measure:
\[P_t(x,dy)=p_t(y-x)dy,\]
where $p_t(x)$ is the density.  It is easy to see that 
\begin{equation*}
\begin{split}
\int_{\mathbb R^d}P_tf(x)dx&=\int_{\mathbb R^d}\int_{\mathbb R^d}f(y)P_t(x,dy)dx=\int_{\mathbb R^d}\int_{\mathbb R^d}f(y)p_t(x-y)dydx\\
&=\int_{\mathbb R^d}\int_{\mathbb R^d}f(x+z)p_t(z)dzdx=\int_{\mathbb R^d}\int_{\mathbb R^d}f(x+z)p_t(z)dxdz\\
&=\int_{\mathbb R^d}f(x)dx.
\end{split}
\end{equation*}
Thus, the Lebesgue measure is an invariant measure for the semigroup $P_t$.\\

Next, for future purposes, we give some results concerning the deterministic PDE with respect to $\mathcal A$. As the proofs are similar to that in \cite{MS10}, we omit them. 
\begin{lemma}
	\label{coefficientf}
	Let $f \in L^2 \left( [0,T] \times \mathbb{R}^d ; \mathbb{R} \right)$ and denote  by $(u^n)_{n\in \mathbb{N}}$ the sequence of solutions of the equations
	\begin{equation*}
	\big(\partial_t  +   \mathcal {A}  \big) u^{n} - n u^{n} + f   = 0, \quad \forall n \in \mathbb{N},
	\end{equation*}
	with final condition $u_T^n = 0$. Then we have
	\begin{equation}
	\label{deterministe:n}
	\int_0^T \mathcal {E}(u_t^n) dt \leq c\left[ \frac{1}{n} \int_0^T \|f_t\|_2^2 dt + \int_0^T e^{-2n (T-t)}
	\|f_t\|_2^2 dt \, \right].
	\end{equation}
\end{lemma}

Obviously, \eqref{deterministe:n} implies that $ \displaystyle  \lim_{n\to \infty} \int_0^T \mathcal E(u^n_t) dt = 0$.
We present a strengthened version of this relation in the next corollary. 
\begin{corollary}\label{coefficientfn}
	Let $f$ and $f^n$, for any $n \in \mathbb{ N}$ be in $L^2([0,T] \times\mathbb{ R}^d ; \mathbb{R})$ such that 	$\displaystyle \lim_{n\to \infty} \int_0^T  \| f^n_t - f_t \|_2^2 dt = 0$.  Then,  the solutions  $(u^n)_{n\in \mathbb{N}}$
	of the equations
	\begin{equation*}
	\big(\partial_t  +   \mathcal A\big) u^{n} - n u^{n} + f^n   = 0,
	\end{equation*}
	with final condition $u_T^n = 0$, satisfy the relation $ \displaystyle  \lim_{n\to \infty} \int_0^T\mathcal E(u^n_t) dt = 0$.
\end{corollary}
%%%%%%%%%%%%%%%%%%%%%%%%%%%%%%%%%%%%%%%%%%%%%%%%%%%%%%%%%%%%%%%%%%%%%%%%%%%%%%%%%%%%
\subsection{Regular measures}\label{measure}

In this section, we shall be concerned with some facts related to the time-space Markov process, with the state space $\left[ 0,T\right[ \times \mathbb{R}^{d},$ 
corresponding to the generator $\partial _{t}+ \mathcal A.$ Its associated semigroup will be denoted by $(\widetilde{P}_{t})_{t>0}$. 
We may express it in terms of the density the semigroup $(P_{t})_{t>0}$ in the following way
\begin{equation*}
\widetilde{P}_{t}\psi \left( s,x\right) =\left\{
\begin{array}{cc}
\int_{\mathbb{R}^{d}}p_{t}\left( x,y\right) \psi \left( s+t,y\right) dy, &
if\ \ s+t<T, \\
0, & otherwise,
\end{array}
\right.
\end{equation*}
where $\psi :\left[ 0,T\right[ \times \mathbb{R}^{d}\rightarrow \mathbb{R}$
is a bounded Borel measurable function, $s\in \left[ 0,T\right[ ,x\in
\mathbb{R}^{d}$ and $t>0.$  So we may also write $\big(\widetilde{P}_t \psi\big)_s = P_t \psi_{t+s}$ if $ s+t < T$. The corresponding resolvent has a density
expressed in terms of the density $p_{t}$ too, as follows
\begin{equation*}
\widetilde{U}_{\alpha }\psi \left( t,x\right) =\int_{t}^{T}\int_{\mathbb{R}
^{d}}e^{-\alpha \left( s-t\right) }p_{s-t}\left( x-y\right) \psi \left(
s,y\right) dyds.
\end{equation*}
or
\begin{equation*}
\big(\widetilde{U}_{\alpha }\psi \big)_t =\int_{t}^{T} e^{-\alpha \left( s-t\right) } P_{s-t} \psi_{s} \,  ds.
\end{equation*}
In particular this ensures that the excessive functions with respect to the
time-space Markov process are lower semicontinuous. In fact we will not use
directly the time space process, but only its semigroup and resolvent. For
related facts concerning excessive functions the reader is refered to \cite{BlumenthalGetoor}, \cite{fot} and \cite
{DM}. Some further properties of this semigroup
are presented in the next lemma. The proof of it is almost the same to that of Lemma 2 in \cite{MS10}. Thus we omit the proof.
\begin{lemma}
The semigroup $( \widetilde{P}_{t}) _{t>0}$ acts as a strongly
continuous semigroup of contractions on the spaces 
$L^{2}([ 0,T[ \times \mathbb{R}^{d}) =L^{2}([ 0,T[\,;L^{2}( \mathbb{R}^{d}))$ and $L^{2}([ 0,T[\, ;H^{1}( \mathbb{R}^{d})).$
\end{lemma}

Now we give the definition of the potentials belonging to $\widetilde{F}$, which appears in our obstacle problem.

\begin{definition}
\label{potential} (i) A function $\psi :\left[ 0,T\right] \times \mathbb{R}%
^{d}\rightarrow \overline{\mathbb{R}}$ is called quasicontinuous provided
that for each $\varepsilon >0,$ there exists an open set, $D_{\varepsilon
}\subset \left[ 0,T\right] \times \mathbb{R}^{d},$ such that $\psi $ is
finite and continuous on $D_{\varepsilon }^{c}$ and
\begin{equation*}
\mathbb{P}^{m}\left( \left\{ \omega' \in \Omega ^{\prime }|\exists \, t\in \left[ 0,T
\right] \ s.t.\left( t,X_{t}\left( \omega' \right) \right) \in D_{\varepsilon
}\right\} \right) <\varepsilon .
\end{equation*}

(ii) A function $u:[ 0,T] \times \mathbb{R}^{d}\rightarrow[0,\infty] $ is called a regular potential, provided that its restriction to 
$[0,T[\times\mathbb{ R}^d$ is excessive with respect to the  time-space semigroup, it is
quasicontinuous, $u\in \widetilde{F}$ and $\lim_{t\rightarrow T}u_{t}=0$ in $L^{2}( \mathbb{R}^{d}) $.
\end{definition}

Observe that if a function $\psi $ is quasicontinuous, then the process $
( \psi _{t}( X_{t}) ) _{t\in[ 0,T]}$ is
c\`adl\`ag and has only inaccessible jumps. Next we will present the basic properties of the regular
potentials. Due to the expression of the semigroup $(\widetilde{P}_{t})_{t>0}$ in terms of the density, it follows that two excessive
functions which represent the same element in $\widetilde{F}$ should
coincide.

\begin{theorem}
\label{potentielregulier}
Let $u\in \widetilde{F}.$ Then $u$ has a version which is a regular
potential if and only if there exists a continuous increasing process $
A=( A_{t}) _{t\in [0,T]}$ which is $( \mathcal{F}_{t}) _{t\in [0,T]}-$adapted and such that 
$A_{0}=0$, $\mathbb{E}^{m}[ A_{T}^{2}] <\infty $ and
\begin{equation}
u_{t}(X_{t})=\mathbb{E}\left[ A_{T}\,\big|\mathcal{F}_{t}\right]
-A_{t},\;\mathbb{P}^{m}\mathbf{-}a.s.,  \tag{$i$}  \label{caf}
\end{equation}
for each $t\in[0,T] .$ The process $A$ is uniquely determined
by these properties. Moreover, the following relations hold:
\begin{equation}
u_{t}\left( X_{t}\right) =A_{T}-A_{t}-\sum_{i=1}^{d}\int_{t}^{T}\nabla
u_{s}\left( X_{s-}\right) dW_{s}-\int_t^T\int_{\mathbb R^d} u_s(X_{s-}+z)-u_s(X_{s-})\tilde N(dz,ds),\ \mathbb{P}^{m}-a.s.\, ,  \tag{$ii$}
\label{retro}
\end{equation}%
\begin{equation}
\left\Vert u_{t}\right\Vert _2^{2}+\int_{t}^{T}\mathcal{E}( u_{s},u_s) \,
ds=\mathbb{E}^{m}\left( A_{T}-A_{t}\right) ^{2}  \, , \tag{$iii$}  \label{energ}
\end{equation}%
\begin{equation}
\left( u_{0},\mathcal{\varphi }_{0}\right) +\int_{0}^{T}  %
\mathcal {E} (u_{s},  \mathcal{\varphi }_{s}) +\big( u_{s},\partial _{s}%
\mathcal{\varphi }_{s}\big) \, ds=\int_{0}^{T}\int_{\mathbb{R}^{d}}%
\mathcal{\varphi }\left( s,x\right) \nu \left( dsdx\right) ,  \tag{$iv$}
\label{faible}
\end{equation}%
for each test function $\mathcal{\varphi \in D}_T,$ where $\nu $ is the
measure defined by%
\begin{equation}
\mathcal{\nu }\left( \mathcal{\varphi }\right) =\mathbb{E}^{m}\int_{0}^{T}\mathcal{%
\varphi }\left( t,X_{t}\right) dA_{t},\ \mathcal{\varphi \in C}_{c}\left( %
\left[ 0,T\right] \times \mathbb{R}^{d}\right) .  \tag{$v$}  \label{mesure}
\end{equation}
\end{theorem}
Before proving Theorem \ref{potentielregulier}, we recall a useful lemma. 
\begin{lemma}\label{Dini}(\cite{DM01}, pp.202)
Let $\{f_n\}$ be a decreasing sequence of positive c\`adl\`ag functions on $[0,\infty]$ which tend poitwisely to $0$ as do their left limits. Then the sequence $\{f_n\}$ converges to $0$ uniformly.
\end{lemma}

\textbf{Proof of Theorem \ref{potentielregulier}}:
The uniqueness of the increasing process in the
representation (i) comes from the uniqueness in the Doob-Meyer
decomposition. Let us now assume that $\overline{u}$ is a regular potential which is a
version of $u.$ We will use an approximation of $\overline{u}$ constructed
with the resolvent. By the resolvent equation one has%
\begin{equation*}
\alpha \widetilde{U}_{\alpha }\overline{u}=\alpha \widetilde{U}_{0}\left(
\overline{u}-\alpha \widetilde{U}_{\alpha }\overline{u}\right) .
\end{equation*}
Set $f^{n}=n(\overline{u}-n\widetilde{U}_{n}\overline{u})
$ and $u^{n}=n\widetilde{U}_{n}\overline{u}=\widetilde{U}_{0}f^{n}.$ Since $%
\overline{u}$ is excessive, one has $f^{n}\geq 0$ and $u^{n},n\in \mathbb{N}%
^{\ast },$ is an increasing sequence of excessive functions with limit $%
\overline{u}.$ In fact $u^{n},n\in \mathbb{N}^{\ast },$ are potentials and
their trajectories are continuous. Moreover, the trajectories $%
t\rightarrow \overline{u}_{t}( X_{t}) $ are c\`adl\`ag on $[0,T[$
by the quasi-continuity of $\overline{u}.$ The process $( u_{t}(X_{t})) _{t\in[0,T[}$ is a super-martingale and, because $
\lim_{t\rightarrow T}u_{t}=0$ in $L^2$,  it is a potential and the trajectories have
null limits at $T$. By quasicontinuity of the functions and the fact that $X$ is quasi-left-continuous, we have $^p(u^n_t(X_t))=u^n_t(X_{t-})$ and $^p(u_t(X_t))=u_t(X_{t-})$, where $^p(\cdot)$ denotes the predictable projection. Therefore, $(u^n-u)(X)$ decreasingly converges to $0$ as do their left limits. Then, Lemma \ref{Dini} implies that this approximation also holds uniformly on the
trajectories, on the closed interval $\left[ 0,T\right] ,$%
\begin{equation*}
\lim_{n\rightarrow \infty }\sup_{0\leq t\leq T}\left\vert u_{t}^{n}\left(
X_{t}\right) -\overline{u}_{t}\left( X_t\right) \right\vert =0,\ \ \mathbb{P}^{m}-a.s..
\end{equation*}%
The function $u^{n}$ solves the equation $\left( \partial _{t}+\mathcal {A}\right)%
u^{n}+f^{n}=0$ with the condition $u_{T}^{n}=0$ and its backward
representation is
\begin{equation*}
u_{t}^{n}(X_{t}) =\int_{t}^{T}f_{s}^{n}( X_{s})ds-\sum_{i=1}^{d}\int_{t}^{T}\partial _{i}u_{s}^{n}( X_{s})
dW_{s}^{i}-\int_t^T\int_{\mathbb {R}^d} u^n_s(X_{s-}+z)-u^n_s(X_{s-})\tilde N(dz,ds).
\end{equation*}
If we set $A_{t}^{n}=\int_{0}^{t}f_{s}^{n}\left( X_{s}\right) ds,$ after
conditioning, this representation gives%
\begin{equation*}
\begin{split}
u_{t}^{n}\left( X_{t}\right)
&=A_{T}^{n}-A_{t}^{n}-\sum_{i=1}^{d}\int_{t}^{T}\partial _{i}u_{s}^{n}\left(
X_{s}\right) dW_{s}^{i}-\int_t^T\int_{\mathbb {R}^d} u^n_s(X_{s-}+z)-u^n_s(X_{s-})\tilde N(dz,ds)\\
&=\mathbb{E}^{m}\left[ A_{T}^{n}|\mathcal{F}_{t}\right]
-A_{t}^{n}.\hspace{9.7cm}(\ast)
\end{split}
\end{equation*}%
%In particular one deduces%
%\begin{equation*}
%\begin{split}
%&u_{0}^{n}\left( X_{0}\right) =\mathbb{E}^{m}\big[ A_{T}^{n}/\mathcal{G}_{0}%
%\big]\\
%&=\mathbb{E}^m \big [A_{T}^{n}-\sum_{i=1}^{d}\int_{0}^{T}\partial _{i}u_{s}^{n}\left(
%X_{s}\right) dW_{s}^{i}+\int_0^T\int_{\mathbb {R}^d} u^n_s(X_{s-}+z)-u^n_s(X_{s-})\tilde N(dz,ds)\big ]
%\end{split}
%\end{equation*}%
By the relation $(*)$, it follows that
\begin{equation*}\begin{split}
&\mathbb{E}^{m}\left( A_{T}^{n}-A_{t}^{n}\right) ^{2} \\=&\,\mathbb{E}^{m}\Big( u_{t}^{n}(X_{t}) +\sum_{i=1}^{d}\int_{t}^{T}\partial _{i}u_{s}^{n}(X_{s}) dW_{s}^{i}
+\int_t^T\int_{\mathbb R^d} u^n_s(X_{s-}+z)-u^n_s(X_{s-})\tilde N(dz,ds)\Big) ^{2}\\
=&\left\|u_{t}^{n}\right\|_2 ^{2}+\int_{t}^{T} \mathcal {E}(u_{s}^{n})ds.\hspace{9.9cm}(\ast \ast)
\end{split}\end{equation*}
A similar relation holds for differences. In particular, one has
\begin{equation*}
\mathbb{E}^{m}( A_{T}^{n}-A_{T}^{k}) ^{2}=\|u_{0}^{n}-u_{0}^{k}\|^{2}+2\int_{0}^{T}\mathcal {E}(u_{s}^{n}-u_{s}^{k})ds.
\end{equation*}

Moreover, the preceding lemma ensures that $\lim_{\alpha
\rightarrow \infty }\alpha \widetilde{U}_{\alpha }=I$ in the space $L^{2}([ 0,T[ ;H^{1}( \mathbb{R}^{d}))$, which implies
\begin{equation*}
\lim_{n\rightarrow 0}\int_{0}^{T}\mathcal {E}(
u_{s}^{n}-u_{s}^{k}) \, \,  ds=0.
\end{equation*}%
These last relations imply that there exists a limit $\lim_{n}A_{T}^{n}=:A_{T}$ in the sense of $L^{2}(\mathbb{ P}^{m}).$

%Let us denote by $M^{n}=\left( M_{t}^{n}\right) _{t\in \left[ 0,T\right]},M=\left( M_{t}\right) _{t\in \left[ 0,T\right] }$ the martingales given by the conditional expectations 
Set $M_{t}^{n}:=\mathbb{E}^{m}\left[ A_{T}^{n}|\mathcal{F}
_{t}\right] ,M_{t}:=\mathbb{E}^{m}\left[ A_{T}|\mathcal{F}_{t}\right] .$ Then
one has $\lim_{n \to \infty }M^{n}=M$ in $L^{2}\left( \mathbb{P}^{m}\right)$ and hence
\begin{equation*}
\lim_{n \rightarrow \infty }\mathbb{E}^{m}\sup_{0\leq t\leq T}\left\vert M_{t}^{n}-M_{t}\right\vert^{2}=0.
\end{equation*}

Then the relation $u_{t}^{n}\left( X_{t}\right) =M_{t}^{n}-A_{t}^{n}$ shows
that the processes $A^{n},n\in \mathbb{N}^{\ast },$ also converge uniformly
on the trajectories to a continuous process $A=\left( A_{t}\right) _{t\in %
\left[ 0,T\right] }.$ The inequality%
\begin{equation*}
\sup_{0\leq t\leq T}\left\vert A_{t}^{n}-A_{t}\right\vert \leq
A_{T}+\left\vert A_{T}^{n}-A_{T}\right\vert ,
\end{equation*}%
ensure the conditins to pass to the limit and get%
\begin{equation*}
\lim_{n\rightarrow \infty }\mathbb{E}^{m}\sup_{0\leq t\leq T}\left\vert
A_{t}^{n}-A_{t}\right\vert ^{2}=0.
\end{equation*}%
Passing to the limit in the relations $(\ast)$ and $(\ast\ast)$ one deduces the relations
(i), (ii) and (iii). 
The rest of the proof are almost the same as that of Theorem 2 in \cite{MS10}, so we omit it. $\hspace{2.1cm}\Box$\\

The following lemma is concerning on the uniqueness of the potential related to a Randon measure via relation (iv). 
For its proof, one can refer to \cite{MS10}. 
\begin{lemma}
Let $u$ be a regular potential and $\nu $ a Radon measure on $\left[ 0,T%
\right] \times \mathbb{R}^{d}$ such that relation (iv) holds. Then one has%
\begin{equation*}
\left( \phi ,u_{t}\right) =\int_{t}^{T}\int_{\mathbb{R}^{d}}\left( \int_{%
\mathbb{R}^{d}}\mathcal{\phi }\left( x\right) p_{s-t}\left( x-y\right)
dx\right) \nu \left( dsdy\right) ,
\end{equation*}%
for each $\phi \in L^{2}\left( \mathbb{R}^{d}\right) $ and $t\in \left[ 0,T%
\right] .$
\end{lemma}

We now introduce the class of measures which intervene in the notion of
solution to the obstacle problem.

\begin{definition}
A nonnegative Radon measure $\nu $ defined on $[ 0,T] \times
\mathbb{R}^{d}$ is called regular if there exists a regular potential $u$ such that the relation (iv) from the above theorem is
satisfied.
\end{definition}

As a consequence of the preceding lemma we see that the regular measures are
always represented as in the  relation (v) of the theorem, with a certain
increasing process. We also note the following properties of a regular
measure, with the notation from the theorem.

\begin{enumerate}
\item A set $D\in \mathcal{B}\left( \left[ 0,T\right] \times \mathbb{R}%
^{d}\right) $ satisfies the relation $\nu(D) =0$ if and only if
$\int_{0}^{T}1_{D}\left( t,X_{t}\right) dA_{t}=0,\ \mathbb{P}^{m}-$a.s..

\item If a set $D\in \mathcal{B}\left( \left] 0,T\right[ \times \mathbb{R}%
^{d}\right) $ is polar, in the sense that%
\begin{equation*}
\mathbb{P}^m\left( \left\{ \omega \in \Omega ^{\prime }|\exists t\in \left[ 0,T\right]
,\left( t,X_{t}\left( \omega \right) \right) \in D\right\} \right) =0,
\end{equation*} then $\nu(D) =0.$
\end{enumerate}

\subsection{Hypotheses}
\label{Hypotheses}
Let $ B = (B_t)_{t\geq 0}$ be a standard $d^1$-dimentional Brownian motion on a probability space $( \Omega,\mathcal{F}^B, \mathbb{P})$. 
Over the time interval $[0,T]$  we define the backward filtration $(\mathcal{F}_{s,T}^{B})_{ s\in [0,T]}$ where $\mathcal{F}^B_{s,T}$ is the completion in $\mathcal{F}^B$ of $\sigma (B_{r}-B_{s};s\leq r\leq T)$.

We denote by $\cH_T$ the space of $H^1(\mathbb{R}^d)$-valued
predictable and $\mathcal{F}^B_{t,T}$-adapted processes   $(u_t)_{0\leq t \leq T}$ such that the trajectories $ t \rightarrow u_t $ are in $\widetilde{F}$ a.s. and
$$\left\| u\right\| _{T}^2 < \infty . $$

In the remainder of this paper we assume that the final condition $\Phi$ is a given function in $L^2 (\mathbb{R}^d)$ and the coefficients
\begin{eqnarray*}
f &:& [0,T]\times \Omega \times  \mathbb{R}^d \times \bbR \times \bbR^{d}
\rightarrow \bbR \;, \\
h&=& (h_1,...,h_{d^1}) \, : \,  [0,T]\times \Omega \times  \mathbb{R}^d \times \bbR \times \bbR^{d}
\rightarrow \bbR^{d^1} \; 
\end{eqnarray*}
are random functions predictable with respect to the backward filtration $(\mathcal{F}_{t,T}^{B})_{ t\in [0,T]}$. We set
\begin{equation*}
f (\cdot,\cdot,\cdot, 0,0):=f^0\quad \mbox{and} \quad   h ( \cdot,\cdot, \cdot ,0,0):=h^0=(h_1^0,...,h_{d^1}^0).
\end{equation*}
and assume the following hypotheses : \\[0.2cm]
 \textbf{Assumption (H)}: There exist  non-negative
 constants $C,\,\beta $ such that
\begin{enumerate}
\item[\textbf{(i)}]
 $ |f_t(\omega,x,y,z) -f( t,\omega, x,y^{\p},z^{\p}) | \leq \, C\big(|
y-y^{\p}| + |z-z^{\p}| \big) $
\item[\textbf{(ii)}] $  \Big(\sum_{j=1}^{d_1}| h_{j,t}(\omega,x,y,z) -h_j( t,\omega,x,y^{\p},z^{\p})
|^2\Big)^{\frac{1}{2}}\leq C \, | y-y^{\p}|+\, \beta \, |z-z^{\p}|,$
\item[\textbf{(iii)}]
  the contraction property (as in \cite{Denis2}) :
$\beta ^{2}< 1 \, $.
\end{enumerate}
\textbf{Assumption (HD2)}%
$$ \mathbb{E} \left(  \left\| f^0\right\|_{2,2}^2+\left\| h^0\right\| _{2,2}^2\right) <+\infty.$$
\textbf{Assumption (HO)} : The obstacle $ v (t,\omega, x)$ is a predictable random function w.r.t the backward filtration $(\mathcal{F}^B_{t,T})$. We also assume  that $ (t,x) \mapsto v(t,\omega, x) $  is $\mathbb{P}$\textbf{-}a.s. quasicontinuous on $[0,T]$ and satisfies $v (T, \cdot) \leq \Phi (\cdot).$\\

We recall that a usual solution (non reflected one) of the equation \eqref{SPDE1}
with final condition $ u_T = \Phi$, is a processus $ u \in \mathcal{H}_T$  such that for each
test function $\varphi \in \mathcal{D}_T$ and any $ t \in [0,T]$, we have a.s.

\begin{equation}
\label{weak:SPDE}
\begin{split}
& \int_{t}^{T }\big[\left( u_{s},\partial_{s}\varphi_s  \right) +
\mathcal {E} (u_{s},  \varphi_s )\big]ds- \big(\Phi,
\varphi_T \big) +
\big( u_t, \varphi_t \big)  = \int_{t}^{T} \left( f_s , \varphi_s \right)   ds
 + \int_{t}^{T} \left( h_s, \varphi_s \right)\cdot \overleftarrow{dB}_{s}.\\
\end{split}
\end{equation}
By Theorem 8 in \cite{Denis2} we have existence and uniqueness of the solution. Moreover, the solution belongs to $\mathcal{H}_T$. We denote by 
$ \mathcal{U }(\Phi, f,h)$ this solution.

\subsection{Quasi-continuity properties}
\label{section:Ito}
In this section we are going to prove the quasi-continuity of the solution of the linear equation, i.e. when  $f$ and $h$ do not depend of $u$ and $ \nabla u$.  To this end we first extend the double stochastic It\^o's formula to our framework. We start by proving the following doubly representation theorem.
\begin{theorem}
\label{FK}
Let $u\in \mathcal{H}_{T}$ be a solution of the equation
\[du_{t}+\mathcal {A}  u_t dt + f_tdt  + h_t\overleftarrow{ dB}_t = 0,\]
where $f,\,h$ are predictable processes such that
$$\mathbb{E } \int_0^T \big[\big \|f_t \big \|_2^2 + \big\|h_t \big \|_2^2 \big] \, dt  < \infty \quad  \mbox{and} \quad \|\Phi \|_2^2 < \infty.$$
Then, for any $0\leq s\leq t\leq T$, one has the following stochastic representation, $\mathbb{P}^{m}$-a.s.,
\begin{equation}
\label{Ito:u}
\begin{split}
u\left( t,X_{t}\right) -u\left( s,X_{s}\right)
=&\sum\limits_{i}\int\limits_{s}^{t}\partial _{i}u\left( r,X_{r}\right)
dW_{r}^{i}+\int_s^t\int_{\mathbb R^d} u_r(X_{r-}+z)-u_r(X_{r-})\tilde N(dz,dr)\\
&-\int\limits_{s}^{t}f_r\left(X_{r}\right) dr-\int\limits_{s}^{t}h_r\left( X_{r}\right) \cdot \overleftarrow{dB}_r.
\end{split}
\end{equation}
\end{theorem}
Noting that $ \mathcal{F}_T$ and $\mathcal{ F}^B_{0,T}$ are independent under $ \mathbb{P}\otimes\mathbb{ P}^m$ and therefore in the above formula the stochastic integrals with respect to $ dW_t$ act independently of $\mathcal{ F}_{0,T}^B$  and similarly the integral with respect to
$\overleftarrow{dB}_t$ acts independently of $ \mathcal{F}_T$ .
\begin{proof} The proof will be similar to that of Proposition 11 in \cite{Denis2}. Thus we only give a sketch of it. The solution can be decomposed into two terms each corresponding to one of the coefficients $f$ and $h$. It is enough to treat them separately.\\
a) In the case where $h=0$, one can use It\^o's formula to get the desired result. Thus we omit the proof.\\
b) In the case where $f=0$, we first establish the formula for the elementary process,
$$h_t(\omega,x)=\chi_{A}(\omega)\chi_{[r_1,r_2)}(t)e(x), $$
where $0\le r_1\le r_2 \le T$, $A \in \mathcal F^B_{r_2,T}$ and $e \in \mathcal D(\mathcal A^{\frac{3}{2}})\cap L^{\infty}$. In this case, let 
$$v^n_t:=\int_{t}^TP_{s-t}h^n_sd\overleftarrow{B}_s,$$
where $h^n_t=2^n(B_{t^n_i}-B_{t^n_{i+1}})h_{t^n_{i+1}}$ for $t \in (t^n_{i},t^n_{i+1}]$ and $t^n_i=\frac{i}{2^n}T$. Then, Lemma 10 in \cite{Denis2} implies that 
$$\lim_{n}E\sup_{0\le t\le T}\mathcal E(v^n_t-u_t)=0.$$
We may have the following relation:
\begin{equation*}
v^n_t(X_t)=-\int^T_{t}h^n(s,X_s)ds+\int^T_t\nabla v^n_s(X_s)dW_s+\int_t^Tv^n_s(X_{s-}+z)-v^n_s(X_{s-})\tilde N(dz,ds).
\end{equation*}
Letting $n \to \infty$, one has
\begin{equation*}
-\int^T_{t}h^n(s,X_s)ds \longrightarrow -\int\limits_{s}^{t}h_r\left( X_{r}\right) \cdot \overleftarrow{dB}_r 
\end{equation*}
\begin{equation*}
\int^T_t\nabla v^n_s(X_s)dW_s \longrightarrow \int^T_t\nabla u_s(X_s)dW_s
\end{equation*}
\begin{equation*}
\int_t^T\int_{\mathbb R^d}v^n_s(X_{s-}+z)-v^n_s(X_{s-})\tilde N(dz,ds) \longrightarrow \int_t^T\int_{\mathbb R^d}u_s(X_{s-}+z)-u_s(X_{s-})\tilde N(dz,ds)
\end{equation*}
Thus we proof the formula for the elementary processes. The general cases can be proved by approximation.
%c) In the case where $f=h=0$, if there exists $l\in L^2([0,T]\times \mathbb R^d)$ such that
%\[\text{div} g=l\] 
%holds in the weak sense, then, according to  Lemma 3.1 in \cite{Stoica}, we have
%\[\int_s^tg_r*dW_r=\int_s^tf(r,X_r)dr.\]
%By Lemma 3.5 in \cite{Stoica}, we have a function $k \in L^2([0,T]\times \mathbb R^d$ such that $k(t,\cdot) \in \mathbb H^1,t-\text{a.s.}$ and $\text{div} g=\Delta k-k$ in the weak sense. Thus the equation can be rewritten as 
%\[du_t+\mathcal Au_t+\Delta k_t-k_t=0.\]
%For $k$ smooth enough, as in the case a), one can show that
%\begin{equation*}
%\begin{split}
%u\left( t,X_{t}\right) -u\left( s,X_{s}\right)
%=&\sum\limits_{i}\int\limits_{s}^{t}\partial _{i}u\left( r,X_{r}\right)
%dW_{r}^{i}+\int_s^t\int_{\mathbb R^d} u_r(X_{r-}+z)-u_r(X_{r-})\tilde N(dz,dr)\\
%&-\int\limits_{s}^{t}(\Delta k-k)\left(r,X_{r}\right) dr
%\end{split}
%\end{equation*}
%Since
%\[\int_s^t\nabla k_r*dW_r=\int_s^t\Delta k(r,X_r)dr,\]
%we can rewrite it as:
%\begin{equation*}
%\label{Ito:u,k}
%\begin{split}
%u\left( t,X_{t}\right) -u\left( s,X_{s}\right)
%=&\sum\limits_{i}\int\limits_{s}^{t}\partial _{i}u\left( r,X_{r}\right)
%dW_{r}^{i}+\int_s^t\int_{\mathbb R^d} u_r(X_{r-}+z)-u_r(X_{r-})\tilde N(dz,dr)\\
%&+\int\limits_{s}^{t}k\left(r,X_{r}\right) dr-\int_s^t\nabla k_r*dW_r
%\end{split}
%\end{equation*}
%Using approaximation method, we shall have that (\ref{Ito:u,k}) holds for  $k \in L^2([0,T]\times \mathbb R^d,k(t,\cdot) \in \mathbb H^1,t-\text{a.s.}$.Moreover, 
%\[\int_s^tg_r*dW_r=\int_s^t\nabla k_r*dW_r-\int_s^t k(r,X_r)dr.\]
%Thus we finish the proof.
\end{proof}

In particular the process $(u_t (X_t))_{t\in[0,T]} $ admits a c\`adl\`ag version which we usually denote by $Y=(Y_t)_{t \in [0,T]}$ and we introduce the notation $ Z_t = \nabla u_t (X_t)$ and $U_t(z)=u_t(X_{t-}+z)-u_t(X_{t-})$. As  a consequence of this theorem,  with the application of It\^o's formula and BDG's inequality, we shall have the following result.

\begin{corollary}
\label{FK2}
Under the hypothesis of the preceding theorem one has the following stochastic representation for $ u^2$, $\mathbb{P}\otimes \mathbb{P}^{m}\textbf{-}a.e.$, for any $0\leq t\leq T$,
\begin{equation}
\label{Ito:2}
\begin{split}
u_t^2\left(X_{t}\right) -\Phi^2\big(X_{T}\big)& = 2 \int_t^T \big[ u_s f_s (X_s) -|Z_s|^2-\int U^2_s(z)v(dz) + \frac{1}{2}|h_s|^2( X_s) \, \big] \, ds \\
&- 2 \sum\limits_{i}\int\limits_{t}^{T} \big(u_r \partial _{i}u_r \big) \left(X_{r}\right)
dW_{r}^{i}-\int_t^T \int_{\mathbb R^d} u^2_t(X_{t-}+z)-u^2_t(X_{t-})\tilde N(dz,dt) \\
 &+ 2 \int_{t}^{T} \big(u_r h_r \big) \left(X_{r}\right) \cdot \overleftarrow{dB}_r.\\
\end{split}
\end{equation}

\begin{equation}
\label{estimationYZ}
\mathbb{E}\mathbb{E}^m \, \big( \sup_{t \leq s \leq T} \left|u_s\right|^2 \big) + \mathbb{E }\Big[
\int_t^T \mathcal E(u_s) \, ds \Big] \leq c  \, \Big[ \|\phi \|_2^2 + \mathbb{ E} \int_t^T \big[ \|f_s\|_2^2 + \|h_s\|_2^2 \, \big]\, ds  \Big],
\end{equation}
for each $ t\in [0,T]$.
\end{corollary}
\begin{proof}Firstly, we may represent the solution in the form
\begin{equation*}
\begin{split}
u_t\left(X_{t}\right) -u_s\left(X_{s}\right)
&=\sum\limits_{i}\int\limits_{s}^{t}\partial _{i}u_r\left(X_{r} \right)
dW_{r}^{i}+\int_s^t\int_{\mathbb R^d}u_r(X_{r-}+z)-u_r(X_{r-})\tilde N(dz,dr)\\
&-\int\limits_{s}^{t} f_r (X_r)  \,  dr  -\int\limits_{s}^{t}h_r (X_{r}) \cdot \overleftarrow{dB}_r.
\end{split}
\end{equation*}
By similar proof in Lemma 1.3 of \cite{PardouxPeng94}, it follows that
\begin{equation*}
\begin{split}
u_t^2\left(X_{t}\right) -u_s^2\left(X_{s}\right) =& - 2 \int_s^t \big[ u_r f_r  (X_r) - |\nabla u_r |^2 (X_r) -|h_r|^2 (X_r) \, \big] \, dr- 2 \int_{s}^{t} \big(u_r h_r \big) \left(X_{r}\right) \cdot \overleftarrow{dB}_r \\
&+\int_s^t\int_{\mathbb R^d}(u_r(X_{r-}+z)-u_r(X_{r-}))^2v(dz)dr+ 2 \sum\limits_{i}\int\limits_{s}^{t} \big(u_r \partial _{i}u_r \big) \left(X_{r}\right)
dW_{r}^{i} \\
&+\int_s^t\int_{\mathbb R^d}u^2_r(X_{r-}+z)-u^2_r(X_{r-})\tilde N(dz,dr).\\
\end{split}
\end{equation*}
%On the other hand, by Lemma 3.1 of \cite{Stoica} one has
%$$
%- 2 \int_s^t div (u_r \, g_r) (X_r) \, dr = \int_s^t u_r \, g_r  * dW_r,
%$$
%so that the preceding relation immediately leads to the relation \eqref{Ito:2}.
Then the standard calculations of BSDE involving Young's inequality, B-D-G inequality and Gronwall's lemma give the estimate \eqref{estimationYZ}.  
%Finally to obtain the result with general $g$ one  proceeds by approximation.
\end{proof}
From Corollary \ref{FK2}, one can easily obtain the following.
\begin{lemma}
	\label{coefficienthn}
	Let $h$ and $h^n$, for $n \in \mathbb{N}$, be $L^2(\mathbb{R}^d; \mathbb{R}^{d^{1}})$-valued predictable processes on $ [0,T]$ with respect to
	$(\mathcal{F}_{t,T}^B)_{t\geq 0}$ and such that
	$$
	\mathbb{E} \int_0^T \|h_t\|_2^2 dt < \infty ,\quad \mathbb{E} \int_0^T \|h_t^n\|_2^2 dt < \infty \quad \mbox{and} \;
	\lim_{n\to \infty} \mathbb{E} \, \int_0^T  \| h^n_t - h_t \|_2^2 dt = 0.
	$$
	Let  $(u^n)_{n\in \mathbb{N}}$ be the solutions
	of the equations
	\begin{equation*}
	d u_t^{n} +  [ \mathcal {A} u_t^{n} - n u_t^{n}]\, dt  + h_t^n \cdot \overleftarrow{dB}_t   = 0,
	\end{equation*}
	with final condition $u_T^n = 0$, for each $n\in \mathbb{N}$. Then, one has
	$ \displaystyle  \lim_{n\to \infty} \int_0^T\mathcal E(u_t^n) dt = 0$.
\end{lemma}

Thanks to estimate \eqref{estimationYZ}, we can do a similar proof as that of Proposition 1 in \cite{MS10} and get the following proposition which concerns the 
quasicontinuity of the solution of an SPDE.  
%%In the deterministic case it was proven in \cite{Stoica} that the solution of a quasilinear equation has a quasicontinuous version. Here we shall prove a similar property for the solution of an SPDE as is stated in the next proposition.
\begin{proposition}
\label{quasicontinuit?EDPS}
Under the hypothesis of Theorem \ref{FK}, there exists  a function $ \bar{u} \, : \, [0,T]\times  \Omega \times \mathbb{R}^d \longrightarrow \mathbb{R}$
  which is a  quasicontinuous version of $u$, in the sense that  for each $\epsilon >0,$ there exits a predictable
  random set $D^{\epsilon} \subset
[0,T] \times \Omega \times \mathbb{R}^d $
such that $\mathbb{P}$\textbf{-}a.s.  the section $ D_{\omega}^{\epsilon}$ is open and $\bar{ u} \left(\cdot, \omega, \cdot \right) $ is continuous on  its  complement  $\left(D_{\omega}^{\epsilon}\right)^c$ and
$$
\mathbb{P}\otimes \mathbb{P}^m \, \left( (\omega, \omega') \, \big| \; \exists t \in [0,T]\; s.t. \;
\left(t, \omega, X_t (\omega') \right) \in
D^{\epsilon} \, \right) \leq \epsilon .
$$
In particular  the process  $\big(\overline{u}_t (X_t) \big)_{t \in[0,T]}$ has c\`adl\`ag trajectories,
$\mathbb{P}\otimes \mathbb{P}^m\textbf{-}a.s.$
\end{proposition}
We also need the quasicontinuity of the solution associated to a random
regular measure, as stated in the next proposition. We first give the formal definition of this object.

\begin{definition}
\label{random:regular:measure}
We say that $ u \in \mathcal{H}_T$ is a random regular potential provided that $ u(\cdot,\omega,\cdot)$ has a version which is regular potential, $\mathbb{P}(d\omega)-a.s.$ The random variable $ \nu\, :\, \Omega \longrightarrow \mathcal{M} \left([0,T] \times \mathbb{R}^d \right) $ with values in the set of regular measures on $ [0,T] \times \mathbb{R}^d$ is called a regular random measure, provided that there exits a random regular potential $u$ such that the  measure $ \nu(\omega) (dtdx)$ is associated to the regular potential $ u (\cdot, \omega, \cdot)$, $\mathbb{P}(d\omega)\textbf{-} a.s.$
\end{definition}
The relation between a random measure and its associated random regular potential is described by the following proposition, the proof of which results from  approximation procedure used in the proof of  Theorem \ref{potentielregulier}.
\begin{proposition}
\label{quasicontinuit?bis}
Let $u$ be  a random regular potential and $\nu $ be the associated random regular measure. Let  $\overline{u}$ be the excessive version of $u,$ i.e. $%
\overline{u}\left( \cdot ,\omega ,\cdot \right) $ is a.s. an $(\widetilde{P}_{t}) _{t>0}$ -excessive function which coincides with $
u\left( \cdot ,\omega ,\cdot \right) ,$ $\ dtdx\textbf{-}a.e.$ Then we have the following
properties:\\[0.2cm]
(i) For each $\varepsilon >0,$ there exists a $( \mathcal{F}
_{t,T}^{B}) _{t\in \left[ 0,T\right] }$ -predictable random set $
D^{\varepsilon }\subset \left[ 0,T\right] \times \Omega \times \mathbb{R}^{d}
$ such that $P$ -a.s. the section $D_{\omega }^{\varepsilon }$ is open and $
\overline{u}\left( \cdot ,\omega ,\cdot \right) $ is continuous on its
complement $\left( D_{\omega }^{\varepsilon }\right) ^{c}$ and
\begin{equation*}
\mathbb{P}\otimes \mathbb{P}^{m}\left( \left( \omega ,\omega ^{\prime }\right) /\exists
t\in \left[ 0,T\right] \ \ s.t.\left( t,\omega ,X_{t}\left( \omega^{\prime} \right)
\right) \in D_{\omega }^{\varepsilon }\right) \leq \varepsilon .
\end{equation*}
In particular  the process  $(\overline{u}_t (X_t))_{t \in[0,T]}$ has  c\`adl\`ag trajectories,
$\mathbb{P}\otimes \mathbb{P}^m-$a.s.\\
(ii) There exists a continuous increasing process $A=\left( A_{t}\right)
_{t\in \left[ 0,T\right] }$ defined on $\Omega \times \Omega ^{\prime }$
such that $A_{s}-A_{t}$ is measurable with respect to the $\mathbb{P}\otimes \mathbb{P}^{m}$-completion 
of $\mathcal{F}_{t,T}^{B}\vee \sigma( W_{r}, r\in[t,s]) $, for any $ 0 \leq s \leq t \leq T$, and such that the following relations are fulfilled
a.s., with any $\mathcal{\varphi }\in \mathcal{D}$ and $t\in \left[ 0,T\right]
,$%
\begin{equation*}
\begin{split}
&(a) \quad \left( u_{t},\mathcal{\varphi }_{t}\right) +\int_{t}^{T}\left( \frac{1}{2}%
\left( \nabla u_{s},\nabla \mathcal{\varphi }_{s}\right) +\left( u_{s},\partial _{s}%
\mathcal{\varphi }_{s}\right) \right) ds=\int_{t}^{T}\int_{\mathbb{R}^{d}}%
\mathcal{\varphi }\left( s,x\right) \nu \left( dsdx\right),\\
& (b)  \quad u_{t}(X_{t})=\mathbb{E}\left[ A_{T}\,\big|\mathcal{F}_{t}\vee \mathcal{F}%
_{t,T}^{B}\right] -A_{t}\,,\\
& (c) \quad u_{t}\left( X_{t}\right) =A_{T}-A_{t}-\sum_{i=1}^{d}\int_{t}^{T}\partial
_{i}u_{s}\left( X_{s}\right) dW_{s}^{i}+\int_0^t u_s(X_{s-}+z)-u_s(X_s) \tilde N(dz,ds),\\
& (d) \quad \left\Vert u_{t}\right\Vert _2^{2}+ \int_{t}^{T}\mathcal E(u_s)
\, ds=\mathbb{E}^{m}\left( A_{T}-A_{t}\right) ^{2},\\
& (e) \quad \mathcal{\nu }\left( \mathcal{\varphi }\right) =\mathbb{E}^{m}\int_{0}^{T}\mathcal{%
\varphi }\left( t,X_{t}\right) dA_{t}.
\end{split}
\end{equation*}

\end{proposition}

%%\begin{proof}
%%The proof of this proposition results from the approximation procedure used
%%in the proof of  Theorem \ref{potentielregulier}.\\[0.2cm]
%%(i) Let $ r > 0$. The process $ \bar{u}^r = (\bar{u}^r_t)_{t \in[0,T]}$, defined by $ \bar{u}^r_t = P_r  u _{t+r},$ has the property
%%that $ (t,x) \longrightarrow \bar{u}^r_t$ is jointly  continuous $\mathbb{P}\textbf{-}a.s$. We also have
%%$$
%%\lim_{r \to 0} \mathbb{E} \mathbb{E}^m \sup_{ 0 \leq t \leq T} \big| \bar{u}^r_t (X_t) - \bar{u}_t (X_t) \big|^2 = 0,
%%$$
%%by the arguments used at the end of the proof of Theorem \ref{potentielregulier}. The one concludes as in the proof of the preceding proposition.\\[0.2cm]
%%(ii) The construction of the increasing process described in Theorem \ref{potentielregulier}  holds globally  for a random regular potential producing on $a.e.$
%%trajectory $ \omega \in \Omega$, the increasing process corresponding to $ u (\cdot,\omega,\cdot)$.
%%\end{proof}

%%We remark that, taking the expectation of the  relation  (ii-d) of this
%%proposition one gets $$\mathbb{E}\mathbb{E}^{m}\left( A_{T}^{2}\right) =\mathbb{E}\big( \left\Vert
%%u_{0}\right\Vert_2 ^{2}+\int_{0}^{T}\mathcal E(u_t) \,  dt\big).$$

%%%%%%%%%%%%%%%%%%%%%%%%%%%%%%%%%%%%%%%%%%%%%%%%%%%%%%%%%%%%%%%%%%%%%%%%%%%%%%%%%
\section{Existence and uniqueness of the solution of the obstacle problem  }\label{existenceetunicite}
\subsection{The weak solution}
We now precise the definition of the solution of our obstacle problem. We recall that the data satisfy the hypotheses of Section \ref{Hypotheses}.
\begin{definition}
\label{o-spde} We say that a pair  $(u,\nu )$ is a weak solution of the obstacle problem for the SPDE
\eqref{SPDE1} associated to $(\Phi,f,h,v)$, if\\[0.2cm]
(i) $ u \in \mathcal{H}_T $ and  $u (t,x)\geq v (t,x)$,
$d\mathbb{P}\otimes dt\otimes dx-a.e.$ and $u(T,x)=\Phi(x)$,  $d\mathbb{P}\otimes dx-a.e.$.\\[0.2cm]
(ii)  $\nu $ is a random regular
measure on $(0,T) \times \mathbb{R}^d$.\\[0.2cm]
  (iii)  for each $\varphi \in \mathcal{D}_T$ and $ t \in [0,T]$,
\begin{equation}
\label{weak:RSPDE}
\begin{split}
\int_{t}^{T } \big[\left( u_{s},\partial_{s}\varphi_s \right) +&
\,\mathcal {E}( u_{s}, \varphi_s ) \big] ds - \big(\Phi,
\varphi_T \big) +
\big( u_t, \varphi_t \big)
 = \int_{t}^{T} \left(f_s \big(u_{s},\nabla u_s \big), \varphi_s \right)  ds\\& + \int_{t}^{T} \left( h_s\left( u_{s},\nabla u_s\right) ,\varphi_s  \right) \cdot \overleftarrow{dB}_{s}
 + \int_{t}^{T}\int_{\mathbb{R}^d}\varphi_s(x)\, \nu (dx,ds).\\
\end{split}
\end{equation}
(iv) If $\overline{u}$ is a quasicontinuous version of $u,$ then one has%
\begin{equation*}
\int_{0}^{T}\int_{\mathbb{R}^{d}}\left( \overline{u}_{s}\left( x\right)
-v_{s}\left( x\right) \right) \nu \left( dsdx\right) =0, \; a.s.
\end{equation*}
\end{definition}
We note that a given solution $u$ can be writen as a sum $ u= u_1+u_2,$ where $u_1$ satisfies a linear equation $ u_1 = \mathcal{U} \big(\Phi, f(u, \nabla u), h(u, \nabla u)\big)$ with $f, h$ determined by $u$, while $u_2$ is the random regular potential corresponding to the measure $\nu$.  By the Propositions \ref{quasicontinuit?EDPS} and \ref{quasicontinuit?bis}, the conditions (ii) and (iii)
imply that the process $u$ always admits a quasicontinuous version, so that
the condition (iv) makes sense. We also note that if $\overline{u}$ is a quasicontinuous version of $u$, then the trajectories of $X$ do not visit the set $\{ \overline{u} < v \}$, $\mathbb{P}\otimes \mathbb{P}^m\textbf{-}a.s.$

Here it is the main result of our paper
\begin{theorem}
\label{maintheorem}
Assume that the assumptions \textbf{(H)},  \textbf{(HD2)} and \textbf{(HO)}  hold.
Then  there exists a unique weak solution of the obstacle problem for the SPDE  \eqref{SPDE1} associated to $(\Phi,f,h,v)$.
\end{theorem}

In order to solve the problem we will use the backward stochastic differential equation  technics. In fact,
we shall follow the main steps of the second proof in \cite{Elk2}, based on the penalization procedure.\\
The uniqueness assertion of  Theorem \ref{maintheorem} results from the following comparison result which can be easily proved.
\begin{theorem}
\label{comparaison}
Let $\Phi ^{\prime },f^{\prime },v^{\prime }$ be similar to $\Phi ,f,v$ and
let $\left( u,\nu \right) $ be the solution of the obstacle problem
corresponding to $\left( \Phi ,f,h,v\right) $ and $\left( u^{\prime },\nu
^{\prime }\right) $ the solution corresponding to $\left( \Phi ^{\prime
},f^{\prime },h,v^{\prime }\right) .$ Assume that the following conditions
hold

\begin{description}
\item[$\left( i\right) $] $\Phi \leq \Phi ^{\prime },\ \ dx\otimes d\mathbb{P}$
-a.e.

\item[$\left( ii\right) $] $f\left( u,\nabla u\right) \leq f^{\prime }\left(
u,\nabla u\right) ,\ \ dtdx\otimes \mathbb{P}$ -a.e.

\item[$\left( iii\right) $] $v\leq v^{\prime },\ \ dtdx\otimes \mathbb{P}$ -a.e.
\end{description}

Then one has $u\leq u^{\prime },\ \ dtdx\otimes \mathbb{P}$ -a.e.
\end{theorem}

\subsection{Approximation by the penalization method}
\label{section:penalization}
For $n\in
\mathbb{N}$, let  $ u^n$ be a  solution  of the following  SPDE
\begin{equation}
\begin{split}
\label{SPDE:n} du^n_t (x) +
\mathcal A u^n_t (x)dt +  f(t,x,u^n_t (x),\nabla u^n_t (x)) dt  &+ n(u^n_t (x) -v_t (x))^{-}\, dt 
 \\&+  h (t,x,u^n_t(x),\nabla u^n_t(x))\, \overleftarrow{dB}_t  = 0 \\
\end{split}
\end{equation}
with final  condition $u^n_T=\Phi$.

Now set $ f_{n}(t,x,y,z)=f(t,x,y,z)+n(y-v_t (x))^{-}$ and $ \nu^n (dt,dx)  := n \big(u_t^n (x) - v_t (x) \big)^{-} dt dx $.
Clearly, for each $n \in \mathbb{N}$, $f_n$ is Lipschitz continuous
in $(y,z)$ uniformly in $(t,x)$. For each $ n\in \mathbb{N}$, Theorem 8 in \cite{Denis2} ensures the existence and uniqueness of a weak  solution $ u^n \in \mathcal{H}_T$ of the SPDE \eqref{SPDE:n} associated
with the data $ (\Phi, f_n, g, h)$. We denote  by $  Y_t^n = u^n (t, X_t ) $, $ Z_n = \nabla u^n (t,X_t) $,$U_t^n(z)=u^n_t(X_{t-}+z)-u^n_t(X_{t-})$ and $S_t = v (t, X_t)$. We shall also assume that $u^n$ is quasicontinuous, so that $Y^n$ is $\mathbb{P}\otimes \mathbb{P}^m\textbf{-}a.e.$ c\`adl\`ag. Then $
 \big(Y^n, Z^n,U^n \big)$  solves the BSDE  associated to the data $ (\Phi, f_n, h)$
\begin{equation}
\begin{split}
\label{BSDE:n}
Y_{t}^n =& \Phi(X_T) +\int_{t}^{T}f_r(X_{r},Y^n_{r},Z^n_{r}) dr + \int_{t}^{T}h_r(X_{r},Y^n_{r},Z^n_{r})\cdot  \overleftarrow{dB}_r\\
& + n \int_t^T (Y_r^n  - S_r^n)^{-}dr -\sum\limits_{i}\int_{t}^{T}Z^n_{i,r}dW_{r}^{i} -\int_t^T\int_{\mathbb{R}^d} U_r^n(z) \tilde N(dz,dr)\, .
\end{split}
\end{equation}
We  define $K_{t}^{n}=n\displaystyle
\int_{0}^{t}(Y_{s}^{n}-S_{s})^{-}ds$ and establish the following lemma.

\begin{lemma}
\label{penalization:estimate1}
The quadruple $(Y^{n},Z^{n},U^n, K^n)$ satisfies the following estimates
\begin{equation}
\label{estimate1:n}
\begin{split}
&\mathbb{E} \mathbb{E}^m \left|Y_t^n \right|^2  +  \lambda_{\epsilon} \mathbb{E } \mathbb{E}^m\int_t^T |Z_r^n|^2 dr +\mathbb{E } \mathbb{E}^m \int_t^T\int_{\mathbb{R}^d} |U^n_s(z)|^2v(dz)ds \\
 \leq&\, c\, \mathbb{ E}\mathbb{ E}^m \big[ |\Phi  (X_T) |^2 +   \int_t^T \left( |f^0_s (X_s) |^2 + |h^0_s (X_s)|^2  \right) ds \big]
 + c_{\epsilon}\,  \mathbb{E} \, \mathbb{E}^m \, \int_t^T |Y_r^n |^2 \, dr\\
 & + c_{\delta}\,  \mathbb{E} \, \mathbb{E}^m \,
\big(\sup_{t \leq r \leq T} |S_r|^2\big)  + \delta \, \mathbb{E} \, \mathbb{E}^m \, \big(K_T^n - K_t^n\big)^2 \\
\end{split}
\end{equation}
where $\lambda_{\epsilon} = 1 - \beta^2 - \epsilon $, $c_{\epsilon}, \, c_{\delta}$ are a positive constants and  $\epsilon>0, \, \delta>0 $ can be chosen small enough such that $\lambda_{\epsilon} > 0 $ .
\end{lemma}

\begin{proof} Applying It\^o 's formula to $(Y^n)^2$, it follows that
\begin{equation}
\label{Ito:BSDEn}
\begin{split}
|Y^n_t|^2&+\int_t^T|Z_s^n|^2ds+\int_t^T\int_{\mathbb R^d} |U^n_s(z)|^2v(dz)ds
=|\Phi(X_T)|^2+2\int_t^TY_sf(s,X_s,Y^n_s,Z^n_s)ds\\&+2\int_t^TY_sdK^n_s
+\int_t^T|h(s,X_s,Y^n_s,Z^n_s)|^2ds+2\int_t^TY^n_sh_s(X_s,Y^n_s,Z^n_s)\cdot\overleftarrow{dB}_s
\\&-2\int_t^TY_s^nZ^n_s dW_s-\int^T_t\int_{\mathbb R^d}|Y^n_{s-}+z|^2-|Y_{s-}^n|^2\tilde N(dz,ds).
\end{split}
\end{equation}
Using  assumption \textbf{(H)} and  taking the expectation in the above equation under $ \mathbb{P} \otimes  \mathbb{P}^m$,  we  obtain
\begin{equation*}
\begin{split}
&\mathbb{E} \mathbb{E}^m   \left| Y_{t}^{n}\right| ^{2}+  \mathbb{E} \mathbb{E}^m \int_{t}^{T}\left| Z_{s}^{n}\right| ^{2}ds ++\mathbb{E} \mathbb{E}^m \int^t_T\int_{\mathbb R^d}|U^n_s(z)|^2v(dz)ds\\
 \leq&\, \mathbb{E} \left| \Phi (X_T) \right|
^{2}+ c_{\varepsilon} \mathbb{E} \mathbb{E}^m \int_{t}^{T} \big[ |f_s^0(X_s)|^{2} + |h_s^0(X_s)|^{2} \big] ds+ c_{\varepsilon} \,\mathbb{ E} \mathbb{E}^m \int_{t}^{T}\left| Y_{s}^{n}\right|^{2}ds\\
 & +   \left( \beta^2 + \varepsilon\right)\,  \mathbb{E} \mathbb{E}^m \int_{t}^{T}\left| Z_{s}^{n}\right| ^{2}ds  +\frac{1}{%
	\gamma }\mathbb{E} \mathbb{E}^m [\sup_{t\leq s \leq T}|S_{s}|^{2}]+\gamma
\mathbb{E}\mathbb{E}^m [(K_{T}^{n}-K_{t}^{n})^{2}]
\end{split}
\end{equation*}
where $\varepsilon, \gamma $ are positive constants and $ c_{\varepsilon}$ is a constant which can be different from line to line.  Finally Gronwall's lemma leads to the desired inequality.
\end{proof}

\begin{lemma}
\begin{equation}
\label{estimation:Kn}
\begin{split}
\mathbb{E}\mathbb{E}^m [(K_{T}^{n}-K_{t}^{n})^{2}]  \leq&\,  c' \Big[ \mathbb{E} \mathbb{E}^m \left| Y_{t}^{n}\right| ^{2} + \|\Phi \|_2^2 \Big] +  c_{\varepsilon} \Big[ 
 \mathbb{E}  \int_{t}^{T} \left[ \|f_s^0\|_2^{2}+ \|h_s^0\|_2^{2} \right] ds\\&+ \mathbb{ E} \mathbb{E}^m \int_{t}^{T} \big[\left| Y_{s}^{n}\right| ^{2}\,
 +  \left| Z_{s}^{n}\right| ^{2}+\int_{\mathbb{R}^d} |U^n_s(z)|^2v(dz)\big]ds \Big].
\end{split}
\end{equation}
\end{lemma}
\begin{proof}
Let $(\widetilde{u}^n)_{n\in \mathbb{N}}$ be the weak solutions
of the following linear type equations
\begin{equation*}
d \widetilde{u}_t^{n} + \mathcal A  \widetilde{u}_t^{n} + h_t \left(u_t^n, \nabla u_t^n \right) \cdot \overleftarrow{dB}_t= 0,
\end{equation*}
with final condition $\widetilde{u}_T^n = 0.$ Set $ \widetilde{ Y}_t^n = \widetilde{u}^n (t, X_t ) $, $ \widetilde{Z}_t^n = \nabla \widetilde{u}^n (t, X_t) $ and $\widetilde U^n_t(z)=\widetilde u^n(t,X_{t-}+z)-\widetilde u^n(t,X_{t-})$.
Then, by estimate \eqref{estimationYZ}, one has
\begin{equation}
\label{estimation:tilde}
\mathbb{E} \mathbb{E}^m  \Big[ \big| \widetilde{Y}_{t}^{n}\big| ^{2} + \int_0^T \big| \widetilde{Z}_{s}^{n}\big|\,  ds +\int_0^T\int_{\mathbb R^d}|\widetilde U^n_t(z)|^2v(dz)dt\Big]  \leq \tilde{c} \Lambda
\end{equation}
where $ \Lambda = \displaystyle  \mathbb{E} \mathbb{E}^m \int_{0}^{T}  |h_s(X_s, Y_s^n, Z_s^n)|^{2} ds$.
Since $ u^n - \widetilde{u}^n$ verifies the equation
\begin{equation*}
\partial_t(u_t^n - \widetilde{u}_t^{n})  + \mathcal A  (u^n - \widetilde{u}_t^{n})  + f_t(u_t^n, \nabla u_t^n) + n(u^n_t -v_t)^{-}\, dt = 0,
\end{equation*}
we have the stochastic representation
\begin{equation*}
\begin{split}
Y_{t}^n - \widetilde{Y}^n_t &  = \Phi\left(X_T\right) +\int_{t}^{T}f_r
\left(X_{r},Y^n_{r},Z^n_{r}
\right) dr + K_T^n - K_t^n -\sum\limits_{i}\int_{t}^{T} \big(Z^n_{i,r}- \widetilde{Z}^n_{i,r} \big)\, dW_{r}^{i}\\
&\quad-\int_t^T\int_{\mathbb R^d}\left(U^n_s(z)-\tilde U^n_t(z)\right)\tilde N(dz,ds), \\
\end{split}
\end{equation*}

from which one obtains the estimate
\begin{equation*}
\begin{split}
\mathbb{E}\mathbb{E}^m [(K_{T}^{n}-K_{t}^{n})^{2}]   \leq&\, c \, \mathbb{E}  \mathbb{E}^m  \Big[\left| Y_{t}^{n}\right| ^{2} +  |\widetilde{ Y}_{t}^{n}| ^{2}  +   \big|\Phi\left( X_T\right)\big|^2 + \int_t^T \big( |f_s^0 (X_s) |^2 + \left| Y_{s}^{n}\right| ^{2}+  \left| Z_{s}^{n}\right|^{2} \big)ds  \\
&+ \int_t^T |\widetilde{Z}_{s}^{n}|^{2} \, ds+\int^T_t\int_{\mathbb R^d}|U^n_t(z)|^2v(dz)dt +\int^T_t\int_{\mathbb R^d}|\widetilde U^n_t(z)|^2v(dz)dt\Big].
\end{split}
\end{equation*}
Hence, using \eqref{estimation:tilde}, we get
\begin{equation*}
\begin{split}
\mathbb{E}\mathbb{E}^m [(K_{T}^{n}-K_{t}^{n})^{2}] \leq&\, c' \, \mathbb{E} \mathbb{E}^m  \left[   \left| Y_{t}^{n} \right| ^{2} +
\left|\Phi (X_T)\right|^2 \,   \right]  + c_{\varepsilon}'  \mathbb{E} \mathbb{E}^m \Big[ \int_t^T \left[ |f_s^0 (X_s) |^2 
+ |h_s^0 (X_s) |^2 \right] \, ds\\&+\int_t^T \big[ \left| Y_{s}^{n}\right| ^{2} +  \left| Z_{s}^{n}\right|^{2}+\int_{\mathbb R^d}|U^n_t(z)|^2v(dz)\big] ds    \, \Big],  \\
\end{split}
\end{equation*}
which gives our assertion. \end{proof}
\begin{lemma}
	\label{mainestimate}
	The quadruple $(Y^{n},Z^{n},U^n, K^n)$ satisfies the following estimate
	\begin{equation*}
	\begin{split}
	&\mathbb{E}\mathbb{E}^m \, \big( \sup_{0 \leq s \leq T} \left|Y_s^n \right|^2 \big) +  \mathbb{E }\mathbb{E}^m  \int_0^T |Z_s^n|^2 \, ds  + \mathbb{E}\mathbb{E}^m \left( K_T^n \right)^2 +\mathbb{E} \, \mathbb{E}^m\int_0^T\int_{\mathbb R^d}|U^n_t(z)|^2v(dz)dt \\
	 \leq& \,c \, \Big[ \|\Phi \|_2^2 + \mathbb{E}\mathbb{E}^m \,
	\big(\sup_{0 \leq s \leq T} |S_s|^2\big)   + \mathbb{E}  \int_0^T \big[ \|f^0_s\|_2^2 + \|h^0_s\|_2^2 \, \big]\, ds \,  \Big] \\
	\end{split}
	\end{equation*}
	where $ c >0$ is a  constant.
\end{lemma}
\begin{proof} From \eqref{estimate1:n} and \eqref{estimation:Kn} we get
\begin{equation*}
\begin{split}
&\left(1 - \delta c' \right) \mathbb{E} \mathbb{E}^m\left|Y_s^n \right|^2  +
 \left( 1  - \beta^2 - \epsilon  - \delta c_{\varepsilon}' \right)   \mathbb{E } \mathbb{E}^m\int_s^T |Z_r^n|^2 \, dr +\mathbb{E}\mathbb{E}^m\int_t^T\int_{\mathbb R^d}|U^n_r(z)|^2v(dz)dr\\
 \leq& \, \left( 1 + c' \delta \right) \|\Phi \|_2^2 +  \left(c_{\epsilon} + \delta c_{\varepsilon}' \right) \Lambda+\left(c_{\epsilon} + \delta c_{\epsilon}'\right)\mathbb{E} \mathbb{E}^m \int_s^T |Y_r^n |^2 \, ds + c_{\delta}\,  \mathbb{E} \, \mathbb{E}^m \,
\big(\sup_{t \leq r \leq T} |S_r|^2\big) , \\
\end{split}
\end{equation*}
where $ \Lambda = \displaystyle  \mathbb{E}\mathbb{E}^m \int_t^T \big[|f^0_s (X_s) |^2 + |h^0_s(X_s)|^2 \big]ds $. It then follows from Gronwall's lemma that
\begin{equation*}
\begin{split}
&\sup_{0 \leq s \leq T}  \mathbb{E} \mathbb{E}^m \, \left( \left|Y_s^n \right|^2 \right) +  \mathbb{E }\mathbb{E}^m  \int_s^T |Z_r^n|^2 \, dr+ \mathbb{E}\mathbb{E}^m\int_0^T\int_{\mathbb R^d}|U^n_t(z)|^2v(dz)dt + \mathbb{E} \mathbb{E}^m \left( K_T^n \right)^2  \\
 \leq& \,c_1  \, \Big[ \|\Phi \|_2^2 +  \mathbb{E} \mathbb{E}^m \,
\big(\sup_{0 \leq r \leq T} |S_r|^2\big)  +  \mathbb{ E}\int_s^T \Big[ \|f^0_r\|_2^2  + \|h^0_r\|_2^2 \, \Big]\, dr \,  \Big]. \\
\end{split}
\end{equation*}
Coming back to the equation  \eqref{BSDE:n} and using Bukholder-Davis-Gundy's inequality and the last estimates we get our statement.
\end{proof}

In order to prove the strong convergence of the sequence $ (Y^n,Z^n, U^n, K^n)$ we shall need the following essential result.

\begin{lemma}
\label{essentiel}
\begin{equation}
\lim_{n \to \infty}  \mathbb{E} \mathbb{E}^m \left[
\sup_{ 0\leq t \leq T} \left( \left(Y_t^n - S_t\right)^- \right)^2 \, \right] = 0. \end{equation}
\end{lemma}
\begin{proof}
Let $(u^n)_{ n\in \mathbb{N}}$  be the sequence of solutions of the penalized SPDE defined in \eqref{SPDE:n}. 
From Lemma \ref{mainestimate}, it follows that the sequence $( f (u^n, \nabla u^n), h (u^n, \nabla u^n))_{n\in \mathbb{N}}$ is bounded in $L^2 ( [0,T] \times \Omega \times \mathbb{R}^d ; \mathbb{R}^{1 +d_1} )$.  By double extraction argument, we may choose a subsequence which is weakly convergent to  a system of predictable processes $(\bar{f}, \bar{h})$. Moreover, we can construct sequences $(\hat{f}^k)_k$ and $(\hat{h}^k)_k$ of convex combinations of elements of the form $$ \hat{f}^k = \sum_{i \in I_k} \alpha_i^k f (u^i,\nabla u^i ), \quad
 \hat{h}^k = \sum_{i \in I_k} \alpha_i^k h (u^i,\nabla u^i ) $$
converging strongly to $\bar{f}$ and $\bar{h}$ respectively in $L^2([0,T]\times\Omega\times\mathbb{R}^d;\mathbb{R})$ and $L^2([0,T]\times\Omega\times\mathbb{R}^d;\mathbb{R}^{d_1})$. \\
For $i \geq n$, we denote by $u^{i,n}$ the solution of the equation
\begin{equation}
\label{SPDE:i}
du_t^{i,n} +   [\mathcal A  u_t^{i,n} - n u_t^{i,n} + n v_t + f_t (u^i,\nabla u^i ) ]dt  +
h_t (u^i,\nabla u^i ) \cdot \overleftarrow{dB}_t = 0
\end{equation}
with final condition $u_T ^{i,n} = v_T$. By comparison theorem, we know that $u^{i,n} \leq u^i$. We set
$ \hat{u}^k =\sum_{i \in I_k} \alpha_i^k  u^{i,n_k},$ where $ n_k= \inf I_k$ and deduce that
\begin{equation}
\label{monotone}
\hat{u}^k \leq \sum_{i \in I_k} \alpha_i^k  u^{i} \leq \lim_{n \to \infty}u^n,
\end{equation}
where the last inequality comes from the monotonicity of the sequence $u^n$.
Moreover, it is clear that $ \hat{u}^k$ is a solution of the equation
\begin{equation}
\label{SPDE:k}
d\hat{u}_t^k + [\mathcal A  \hat{u}_t^k - n_k \hat{u}_t^k + n_k v_t + \hat{f}_t^k    \, ] \, dt +
\hat{h}_t^k \cdot \overleftarrow{dB}_t = 0
\end{equation}
with final condition $\hat{u}_T^k = v_T$.\\
Now we are going to take the advantage of the fact that the equations satisfied by  the sequence of solutions $\hat{u}^k$ have strongly convergent coefficients. Let us denote by $\widehat{Y}^k$ the c\`adl\`ag version on $[0,T]$ of the process
$(\hat{u}^k (X_t))_{t \in[0,T]}$, for any $ k \in \mathbb{N}$.
We will prove now that there exists a subsequence such that, for any stopping time $\tau \le T$
\begin{equation}
\label{Ychapeau}
\displaystyle \lim_{ k \to \infty} \widehat{Y}^k_\tau - S_\tau =0, \;\;  \mathbb{P}\otimes \mathbb{P}^m\textbf{-}a.s.
\end{equation}
Since equation \eqref{SPDE:k} is linear, the solution  decomposes as a sum of three terms each corresponding to one of the coefficients $\hat{f}^k,\hat{h}^k,v$. So it is enough to treat separately each term.\\[0.2cm]
a) In   the case where $f \equiv 0$ and $h \equiv 0$, one obtains the term corresponding to $\nu$. The equation becomes:
\begin{equation*}
\big(\partial_t  +  \mathcal A  \big) \, {\hat u}^{k} - k {\hat u}^{k} + k v   = 0.
\end{equation*}  
It is well known that $\widehat{Y}^k$ has the following representation: 
\begin{equation*}
\widehat{Y}_\tau^k = \mathbb{E}^m\Big[e^{-k(T-\tau)} S_T  + k \int_\tau^T e^{-k (r -\tau)} S_rdr \, \Big| \mathcal{G}_\tau\Big].
\end{equation*}
Then, it is easy to get that $ \lim_{ k \to \infty} \widehat{Y}^k_\tau - S_\tau =0$.\\[0.2cm]
b) In the case where $v \equiv 0$, $h \equiv 0$, the representation of $\widehat{Y}^k$ is given by
\begin{equation}
\label{Ito:Ykf}
\begin{split}
\widehat{Y}_t^k =&\int_t^T e^{-n_k (s-t)} \hat{f}^k_s (X_s) ds -  \int_{t}^{T} e^{-n_k (s-t)}
\nabla\hat{u}_s^k \left(X_{s-}\right)
dW_{s}\\&-\int_t^T\int_{\mathbb R^d} e^{-n_k (s-t)}\big(\hat u^k_s(X_{s-}+z)-\hat u^k_s(X_{s-})\big)\tilde N(dz,ds).
\end{split}
\end{equation}
Since
\begin{equation*}
\Big| \int_t^T e^{-n_k (s-t)} \hat{f}_s^k (X_s) ds \Big| \leq  \frac{1}{\sqrt{2n_k}} \Big(\int_t^T
(\hat{f}_s^k (X_s))^2 ds \Big)^{1/2},
\end{equation*}
then $\displaystyle \lim_{k\to \infty} \sup_{0\leq t \leq T} \Big|\int_t^T e^{-n_k (s-t)} \hat{f}^k_s (X_s) ds\Big| =0,\; \mathbb{P}\otimes \mathbb{P}^m\textbf{-}a.s.$.
 For the second and the third term in the expression of
$\hat{Y}^k$ we make an integration by parts formula and get 
\begin{equation*}
\begin{split}
&\int_{t}^{T} e^{-n_k (s-t)}
\nabla\hat{u}_s^k \left(X_{s-}\right)
dW_{s}+\int_t^T\int_{\mathbb R^d} e^{-n_k (s-t)}\big(\hat u^k_s(X_{s-}+z)-\hat u^k_s(X_{s-})\big)\tilde N(dz,ds)\\
=&\,e^{-n_k (T-t)}U_T^{k} - U_t^{k} + n_k \int_t^T U_s^{k} e^{-n_k(s-t)} ds
\end{split}
\end{equation*}
where $ \displaystyle U_t^{k} =\int_{0}^{t}
\nabla\hat{u}_s^k \left(X_{s-}\right)
dW_{s}+\int_0^t\int_{\mathbb R^d} (\hat u^k_s(X_{s-}+z)-\hat u^k_s(X_{s-})\tilde N(dz,ds)$. According to Corollary  \ref{coefficientfn},   we know that the martingales $U^{k},\, k \in \mathbb{N}$ converges to zero in $L^2,$ and hence on a subsequence we have $ \lim_{k\to \infty} \sup_{0 \leq t \leq T} |U_t^{k} | = 0, \;  \mathbb{P}\otimes \mathbb{P}^m-$a.s..
Therefore, the desired result \eqref{Ychapeau} holds  also in this case.\\[0.2cm]
c) In the case where $f \equiv 0$,  $v \equiv 0$, the representation of $\widehat{Y}^k$ is given by
\begin{equation*}
\label{Ito:Ykg}
\begin{split}
\widehat{Y}_t^k & =  \int_t^T e^{-n_k (s-t)} \hat{{h}}^k_s(X_s) \cdot \overleftarrow{dB}_s-\int_{t}^{T} e^{-n_k (s-t)}
\nabla\hat{u}_s^k \left(X_{s-}\right)
dW_{s}\\
&\quad\quad-\int_t^T\int_{\mathbb R^d} e^{-n_k (s-t)}\big(\hat u^k_s(X_{s-}+z)-\hat u^k_s(X_{s-})\big)\tilde N(dz,ds)\,.
\end{split}
\end{equation*}
On account of Lemma \ref{coefficienthn}, the same arguments used in the previous cases work again.
\vspace{2mm}

Since $\widehat {Y}^n$ is increasing and bounded, it converges to a limit process $Y$. Due to (\ref{monotone}) and (\ref{Ychapeau}), we see that $Y_\tau \ge S_\tau$ for any stopping time $\tau$. From this and the section theorem (\cite{DM01}, pp.220), we deduce that $\forall t,\ Y_t \ge S_t,\ \mathbb{P}\otimes \mathbb{P}^m-$a.s.. Hence, $ \forall t, \ (Y^n_t-S_t)^- \searrow 0,\ \mathbb{P}\otimes\mathbb{P}^m-$a.s.. Since $Y^n \nearrow Y$, if we denote by $^pX$ the predictable projection of $X$, then $^pY^n \nearrow {^pY}$ and ${^pY} \ge {^pS}$. But the jumping times of $Y^n$ are all inaccessible. Henceforth $^pY^n=Y^n_{-}$, where $Y^n_{-}$ is the left limit process of $Y^n$. In the same way, we have $^pS=S_{-}$. So we have proved that $^pY^n \nearrow {^pY} \ge {^pS}$, which implies that $(Y^n_{t-}-S_{t-})^- \searrow 0, \mathbb{P}\otimes \mathbb{P}^m-$a.s. for any $t$. Consequencely from  Lemma \ref{Dini}, we deduce that $\sup_{0 \le t \le T}(Y^n_t-S_t)^{-} \longrightarrow 0, $ a.s.. Finally, we finish the proof of the lemma, due to dominated convergence theorem. 
\end{proof}

Then, we have the following result.
\begin{lemma}
	\label{convergence:YZK}
	There exists a predictable measurable  set of processes
	$ \left(Y_t, Z_t,U_t, K_t \right)_{t \in [0,T]}$ such that
	\begin{equation}
	\label{convergence:YZKn}
	\begin{split}
\mathbb{E}\mathbb{E}^m \Big[ \sup_{0 \leq s \leq T}  |Y_t^n - Y_t|^2&+ \int_0^T  |Z_t^n - Z_t|^2  dt  +\sup_{ 0 \leq t \leq T} |K_t^n - K_t|^2\\ &+ \int_0^T \int_{\mathbb{R}^d} |U^n_t(z)-U_t(z)|^2v(dz)dt\Big] \longrightarrow 0 \quad  \mbox{as} \; \,     n \to \infty.
	 \end{split}
	  \end{equation}
	Moreover, we have that $(Y_t, Z_t,U_t, K_t)_{t \in [0,T]}$  satisfies  $Y_t\geq S_t,\; \forall \, t \in [0,T] $ and   $ \displaystyle \int_{0}^{T}(Y_s-S_s)dK_s =  0$, $ \mathbb{P}\otimes \mathbb{P}^m-$a.e..
\end{lemma}
 \begin{proof} From the monotonicity of the sequence $ (f_n)_{n \in \mathbb{N}}$ and comparison theorem (see Theorem \ref{comparaison}), we get  that $ u^n (t,x) \leq u^{n+1} (t,x),\,dtdx\otimes \mathbb{P}$-a.e., therefore one has  $ Y_t^n \leq Y_t^{n+1}$, for all $ t \in [0,T]$, $\mathbb{P}\otimes \mathbb{P}^m$ -a.s. Thus, there exists a predictable real valued process
$ Y = \left(Y_t \right)_{t \in [0,T]}$ such that $ Y_t^n \uparrow Y_t,$ for all $t \in [0,T]$ a.s. and by Lemma \ref{mainestimate} and Fatou's lemma , one gets  $$ \mathbb{E}\mathbb{E}^m \Big( \sup_{0 \leq s \leq T}  |Y_t|^2 \Big) \leq  c.$$
Moreover, from the dominated convergence theorem one has
\begin{equation}
\label{convergence:Y}
\mathbb{E}\mathbb{E}^m  \int_0^T  |Y_t^n - Y_t|^2  dt  \longrightarrow 0 \quad  \mbox{as} \; \,     n \to \infty.\end{equation}
Moreover, for $n \geq p$,
\begin{equation}
\label{Ito:BSDEn}
\begin{split}
& |Y_t^n -Y_t^p|^2+\int_t^T|Z_s^n-Z_s^p|^2 ds+\int_t^T\int_{\mathbb R^d} |U^n_s(z)-U^p_s(z)|^2v(dz)ds\\ 
 = &\,2 \int_t^T(Y_s^n - Y_s^p)[f_s ( X_s, Y_s^n, Z_s^n ) - f_s ( X_s, Y_s^p, Z_s^p )] ds  +  2 \int_t^T(Y_s^n - Y_s^p) d(K_s^n - K_s^p) \\
&+ \int_t^T |h_s(X_s,Y_s^n, Z_s^n )- h_s(X_s,Y_s^p, Z_s^p)|^2  ds-2 \sum\limits_{i}\int_{t}^{T}(Y_s^n - Y_s^p)(Z_{i,s}^n - Z_{i,s}^p)dW_{s}^{i}\\
&+ 2 \int_{t}^{T}(Y_s^n - Y_s^p)[ h_s(X_s, Y_s^n, Z_s^n) - h_s(X_s, Y_s^p, Z_s^p)] \cdot \overleftarrow{dB}_s \\
& +\int_t^T\int_{\mathbb R^d}|Y^n_{s-}+U^n_{s}-Y^p_{s-}-U^p_s|^2-|Y^n_{s-}-Y^p_{s-}|^2\tilde N(dz,ds) .\\
\end{split}
\end{equation}
By standard calculation one deduces that
\begin{equation}
\label{cauchy:Z}
\begin{split}
&\mathbb{E}\mathbb{E}^m  \int_t^T |Z_s^n-Z_s^p|^2 ds+ \mathbb{E}\mathbb{E}^m  \int_t^T\int_{\mathbb R^d} |U^n_s(z)-U^p_s(z)|^2v(dz)ds\\
 \leq& \, c\, \mathbb{E}\mathbb{E}^m  \int_t^T  |Y_s^n - Y_s^p|^2\,
+  4 \mathbb{E} \mathbb{E}^m  \int_t^T \left(Y_s^n -  S_s\right)^- dK_s^p  + 4 \mathbb{E}\mathbb{E}^m  \int_t^T \left(Y_s^p -  S_s\right)^-  \, dK_s^n\\
\end{split}
\end{equation}
Therefore , by  Lemma \ref{essentiel}, \eqref{convergence:Y} and \eqref{cauchy:Z},  it follows that
\begin{equation}
\label{cauchy:YZ}
\begin{split}
&\mathbb{E}\mathbb{E}^m  \int_0^T  |Y_t^n - Y_t^p|^2  dt  + \mathbb{E} \mathbb{E}^m  \int_0^T  |Z_t^n - Z_t^p|^2  dt + \mathbb{E} \mathbb{E}^m  \int_0^T\int_{\mathbb R^d} |U^n_t(z)-U^p_t(z)|^2v(dz)dt \\
& \longrightarrow 0 \quad  \mbox{as} \; \,     n,p \to \infty.
 \end{split}
 \end{equation}
Then we can do a same argument as in Hamad\`ene and Ouknine \cite{HO} and obtain \eqref{convergence:YZK}. 
%We get that there exists a pair $(Z,K,U)$ of progressively measurable processes with values in $ \mathbb{R}^d \times\mathbb{ R}$ such that \begin{equation*}
%%\label{convergence:YZKn}
%\begin{split}
%&\mathbb{E} \, \mathbb{E}^m  \left[ \sup_{0 \leq s \leq T}  |Y_t^n - Y_t|^2+ \int_0^T  |Z_t^n - Z_t|^2  dt  +   \sup_{ 0 \leq t \leq T} |K_t^n - K_t|^2+\int_0^T\int_{\mathbb R^d} |U^n_t(z)-U_t(z)|^2v(dz)dt \right]\\
%&\hspace{-2mm} \longrightarrow 0 \quad \quad  \mbox{as} \; \,     n \to \infty\end{split}\end{equation*}
It is obvious  that $(K_t)_{t \in[0,T]}$ is an increasing continuous process.

Moreover, Lemma \ref{essentiel} implies that 
%On the other hand,  since by Lemma \ref{essentiel}, we have $\lim_{n\rightarrow
%	\infty}\displaystyle \mathbb{E} \mathbb{E}^m \big[\sup_{0 \leq t \leq
%	T}((Y^n_t-S_t)^-)^2\big]=0$;  then,
\begin{equation}
\label{condition:obstacle}
Y_t\geq S_t, \quad \forall \, t \in [0,T], \quad \mathbb{P}\otimes \mathbb{P}^m \mathrm{-} a.s.,
\end{equation}
which yields that $\int_{0}^{T}(Y_s-S_s)dK_s \geq 0$.  

 Finally we also have
$\int_{0}^{T}(Y_s-S_s)dK_s = 0$ since on the other hand the
sequences $(Y^n)_{n\geq 0}$ and $(K^n)_{n\geq 0}$ converge
uniformly (at least for a subsequence) respectively to $Y$ and $K$
and
$$ \int_{0}^{T}( Y_{s}^{n}-S_{s})
dK_{s}^{n}=-n\int_0^{T}(( Y_{s}^{n}-S_{s}) ^{-}) ^{2}ds\leq 0.\,\,
$$
\end{proof}

\textbf{Proof of Theorem \ref{maintheorem}: } 
Since
\begin{equation*}
\begin{split}
\int_0^T\| u_t^n - u_t^p\|_2^2 +\mathcal E(u^n_t-u^p_t)dt
=\mathbb{E}^m \int_0^T[| Y_t^n - Y_t^p|^2 +| Z_t^n - Z_t^p|^2+\int_{\mathbb R^d} |U^n_t(z)-U^p_t(z)|^2v(dz)]dt,
\end{split}
\end{equation*}
by the preceding lemma, one deduces that the sequence $(u^n)_{n\in \mathbb{N}}$ is a Cauchy sequence in $L^2(\Omega \times [0,T]; H^1 (\mathbb{R}^d))$ 
and hence has a limit $u$ in this space. Also from the preceding lemma it follows that $dK_t^n$ weakly converges to  
$dK_t$, $\mathbb{P}\otimes \mathbb{P}^m-a.e.$. This implies that
\begin{equation*}
\lim_{n} \int_0^T \int_{\mathbb{R}^d} n( u^n_t - v_t)^- \varphi (t,x) \, dt dx =  \lim_{n} \mathbb{E}^m \int_0^T  \varphi_t(X_t)dK_t^n = \int_0^T \int_{\mathbb{R}^d} \varphi (t,x)\nu(dt,dx),
\end{equation*}
where $\nu$ is the regular measure defined by
\begin{equation*}
\int_0^T \int_{\mathbb{R}^d} \varphi (t,x) \, \nu\big(dt dx\big) = E^m \int_0^T \varphi_t(X_t) \,dK_t .
\end{equation*}
Writing the equation \eqref{SPDE:n}  in the weak form and passing to the limit one obtains the equation \eqref{weak:RSPDE} with $u$ and this  $\nu$. The arguments we have explained after Definition \ref{o-spde} ensure that $u$ admits a quasicontinuous version $\bar{u}$. Then one deduces that $( \bar{u}_t(X_t))_{t \in [0,T]}$ should coincide with $(Y_t)_{t \in [0,T]}$,  $\mathbb{P}\otimes \mathbb{P}^m-$a.e.. Therefore, the inequality  $ Y_t \geq S_t$ implies $ u \geq v $, $dt \otimes \mathbb{P} \otimes dx-$a.e. and the relation $ \int_0^T(Y_t - S_t)dK_t = 0$ implies the relation $(iv)$ of Definition \ref{o-spde}. $\hfill \Box$


\begin{thebibliography}{99}

\bibitem{BCEF}  Bally, V.,  Caballero, E., El-Karoui, N. and
Fernandez, B. : Reflected BSDE's PDE's and Variational Inequalities.  INRIA report (2004).

%\bibitem{BM01}  Bally V. and Matoussi, A.  : {Weak solutions for
%SPDE's and Backward Doubly SDE's}. \textit{Journal of Theoretical Probability} \textbf{14},  125-164 (2001).

%\bibitem{BensoussanLions78} Bensoussan  A. and Lions J.-L. Applications des Inéquations variationnelles en contrôle
%stochastique. \textit{Dunod}, Paris (1978).

\bibitem{BlumenthalGetoor} Blumenthal, R.M,  and Getoor, R.K. :  Markov processes and potential theory, \textit{Academic Press} (1968).

%\bibitem{Brezis} Brézis H. : Analyse fonctionnelle - Théorie et application. \textit{Dunod} (2005).

\bibitem{DMZ}  Dalang, R.C., Mueller, C. and  Zambotti, L.:
Hitting properties of parabolic SPDE¡¯s with reflection
\textit{Ann. Probab.}, {\bf34} (4) , 1423¨C1450 (2006).

\bibitem{DM01} Dellacherie, C. and Meyer, P.-A.:  Probabilit\'{e}s et Potentiel, vol.B, chap. V \`a VIII, théorie des martingalesXVI, Théorie du potentiel. \textit{Hermann }(1980).

\bibitem{DM} Dellacherie, C. and Meyer, P.-A.:  Probabilit\'{e}s et Potentiel, Chapitres XII \`a XVI, Théorie du potentiel,
Hermann, Paris, 1980. English translation : Potential theory, Chapters XII-XVI, \textit{North-Holland }(1982).


\bibitem{Denis1} Denis, L.  : Solutions of stochastic partial differential equations considered as Dirichlet processes.
 \textit{Bernoulli J.of Probability} \textbf{10}, 783-827 (2004).

\bibitem{Denis2}  Denis, L. and Stoica, I. L. : {A general analytical
result for non-linear s.p.d.e.'s and applications}, \textit{Electronic
Journal of Probability}  \textbf{9},  674-709 (2004).

\bibitem{DMS}  Denis, L., Matoussi A. and Stoica, I. L.  : {$L^p$
estimates for the uniform norm of solutions of quasilinear SPDE's}.
\textit{Probability Theory Related Fileds }\textbf{133},  437-463 (2005).

\bibitem{DMZ12}  Denis, L., Matoussi, A. and Zhang, J.: The obstacle problem for quasilinear stochastic PDEs: Analytical approach.
\textit{The Annals of Probability}, {\bf 42}, 865--905 (2014).


\bibitem{Donati-Pardoux} Donati-Martin, C. and Pardoux E.  :
{White noise driven SPDEs with reflection}. \textit{Probab. Theory Related
Fileds} \textbf{95},  1-24 (1993).

%\bibitem{Elk1} El Karoui, N.  Backward stochastic differential equations a
%general introduction, in vol. Backward Stochastic Differential Equations,
%Pitman Res. Notes in Math. Ser.364 \emph{Longman} (1997).

\bibitem{Elk2}  El Karoui, N.,  Kapoudjian, C.,
Pardoux, E., Peng, S. and  Quenez, M.C. :  Reflected Solutions of
Backward SDE and Related Obstacle Problems for PDEs.  \textit{Annals of  Probability}
\textbf{25},  702-737 (1997).

\bibitem{fot} Fukushima, M.,  Oshima, Y. and Takeda, M. : Dirichlet Forms and
Symmetric Markov Processes, \emph{Walter de Gruyter}, Berlin- New York (1994).

 \bibitem{HO} Hamad\`ene, S. and Ouknine, Y. : Reflected backward stochastic differential equation with jumps and random obstacle,
\textit{Electronic Journal of Probability}\textbf{ 8}, 1-20 (2003) .

\bibitem{HOT} Harraj, N., Ouknine, Y. and Turpin, I.: Double-barriers-reflected BSDEs with jumps and viscosity solutions of parabolic integrodifferential PDEs. \textit{J. Appl. Math. Stoch. Anal.}\textbf{1}, 37-53 (2005)

%\bibitem{Krylov}  Krylov N. V.  : {An analytic approach to SPDEs}. Six
%Perspectives, \textit{AMS Mathematical surveys an Monograph}s, 64, 185-242 (1999).


%bibitem{ls1}  T.Lyons, L.Stoica: On the limit of stochastic integrals of
%differential forms, 10th Winter School on Stochastic Processes and their
%Applications in Siegmundsberg, 13-19 March 1994, in vol.Stochastic Proc.
%Rel. Topics, Stochastics Monographs, \textbf{10}, Gordon and Breach.1996.
%
%\bibitem{ls2}  T.Lyons, L.Stoica: The limits of stochastic integrals of
%differential forms. Ann. of Proba.\textbf{27}, no.1, 1-49, 1999.

%\bibitem{Kunita}  Kunita H. (1994) : {Generalized Solutions of a Stochastic
%Partial Differential Equation}. J. of Theoret. Probab, {7},
%279-308.
%
%\bibitem{MS03}  Matoussi A. and Scheutzow, M.  : {Semilinear
%Stochastic PDE's with nonlinear noise and Backward Doubly SDE's}.
%\textit{J. of Theoret. Probab.} \textbf{15},  1-39 (2002).

\bibitem{MS10} Matoussi, A. and Stoica, L. : { The obstacle problem for quasilinear stochastic PDE'S}. \textit{Annals of Probability.} \textbf{38(3)}, 1143-1179(2010).

%\bibitem{MX1}  Matoussi, A.  Xu, M. : {Sobolev solution for semilinear PDE with obstacle under monotonicity condition}.
%\textit{Electronic Journal of Probability} \textbf{13}, 1035-1067 (2008).

%\bibitem{MX2} Matoussi, A.,  Xu, M.  : {Reflected Backward doubly SDE and Obstacle problem for semilinear Stochastic PDE's}. Forthcoming paper.

%\bibitem{MignotPuel75} Mignot, F  Puel, J.-P.  :  Solutions maximum de certaines inéquations d'évolution paraboliques et inéquations quasi-variationnelles paraboliques. \textit{C.R.A.S.} 280 série A, page 259 and ARMA (1976).

\bibitem{NualartPardoux}  Nualart, D. and Pardoux, E.  : {White noise driven
quasilinear SPDEs with reflection}. \textit{Probab. Theory Related Fields} \textbf{93}, 77-89 (1992).

%\bibitem{PardouxPeng90}  Pardoux, E.  Peng, S. : Adapted solution of a backward stochastic
%differential equation, \emph{Systems and Control Letters}, \textbf{14},
%p. 55-61 (1990).

%\bibitem{PardouxPeng92}  Pardoux, E.  Peng, S.  : Backward SDEs and quasilinear PDEs, in
%Stochastic partial differential equations and their applications,
%B.L.Rozovskii \& R. Sowers eds., LNCIS 176, Springer (1992).

\bibitem{PardouxPeng94}  Pardoux, E. and  Peng, S.  : Backward doubly stochastic differential
equations and systems of quasilinear SPDEs. \textit{Probab. Theory Related Fields} \textbf{98}, 209-227 (1994).

%\bibitem{Pardoux1} Pardoux, E.   : {Stochastic partial differential
%equations and filtering of diffusion process}. \textit{Stochastics} \textbf{3}, 127-167 (1979).

%\bibitem{Revuz}  Revuz, D. and Yor, M.(1999) : {Continuous Martingales and
%Brownian Motion}. Springer, third edition.
\bibitem{Pierre} Pierre, M.: Probl\`emes d'Evolution avec Contraintes Unilaterales et
Potentiels Parabolique. \textit{Comm. in Partial Differential Equations},
{\bf 4(10)}, 1149-1197 (1979).

\bibitem{PIERRE} Pierre, M. : Repr\'esentant Pr\'ecis d'Un Potentiel Parabolique. \textit{S\'eminaire
	de Th\'eorie du Potentiel}, Paris, {\bf No.5}, Lecture Notes in Math. 814, 186-228 (1980).

%\bibitem{Stoica}  Stoica, L. : A probabilistic interpretation of the divergence and BSDE's. \textit{Stochastic Process and their Applications} \textbf{103}, 31-55 (2003).


\bibitem{XuZhang} Xu, T.G. and Zhang,
T.S.: White noise driven SPDEs with reflection: Existence,
uniqueness and large deviation principles.
\textit{Stochatic processes and their applications}, {\bf 119}, 3453-3470 (2009).

\bibitem{Zhang} Zhang, T.S.:
White noise driven SPDEs with reflection: Strong Feller properties and Harnack inequalities,
\textit{Potential Analysis}, {\bf 33} (2),  137-151 (2010).






%\bibitem{Rogers}  Rogers, L. C. G. and Williams, D. (2000) : {Diffusions,
%Markov Processes and Martingales}, volume 2, It\^o Calculus.


%\bibitem{ROZ}  Rozovskii B.L. (1990) :{\ Stochastic Evolution Systems},
%Kluver, Dordrecht- Boston- London.
%
%
%\bibitem{Walsh}  Walsh, J.B. (1986) : {An introduction to stochastic partial
%differential equations}. Ecole d'Et\'e de St-Flour XIV, 1984,
%Lect. Notes in Math 1180, Springer Verlag.
\end{thebibliography}
\end{document}